\documentclass[a4paper,10pt, oneside]{article}
\usepackage{amsmath, amssymb, amsthm}
\usepackage[font=small,labelfont=md,textfont=it]{caption}
\usepackage{floatrow}
\usepackage[titletoc, title]{appendix}
\usepackage[colorlinks,linkcolor=blue,citecolor=blue]{hyperref}
\usepackage{longtable}
\usepackage{diagbox}
\usepackage{booktabs,makecell,multirow}
\usepackage[capitalise, nosort]{cleveref}
\usepackage{cases,color}
\crefname{equation}{}{}
\crefname{lem}{Lemma}{Lemmas}
\crefname{thm}{Theorem}{Theorems}

\DeclareMathOperator{\D}{D}
\DeclareMathOperator{\I}{I}

\newcommand{\dual}[1]{\left\langle {#1} \right\rangle}
\newcommand{\jmp}[1]{{[\![ {#1} ]\!]}}
\newcommand{\nm}[1]{\left\lVert {#1} \right\rVert}
\newcommand{\snm}[1]{\left\lvert {#1} \right\rvert}
\newcommand{\ssnm}[1]
{
  \left\vert\kern-0.25ex
  \left\vert\kern-0.25ex
  \left\vert
  {#1}
  \right\vert\kern-0.25ex
  \right\vert\kern-0.25ex
  \right\vert
}

\numberwithin{equation}{section}

\newtheorem{Def}{Definition}[section]

\newtheorem{lem}{Lemma}[section]
\newtheorem{rem}{Remark}[section]
\newtheorem{thm}{Theorem}[section]

\begin{document}

\title{
  \Large \bf A space-time finite element method for fractional wave problems
\thanks
{
	This work was supported in part by National Natural Science Foundation
	of China (11771312).
}}
\author{
	Binjie Li \thanks{Email: libinjie@scu.edu.cn},
	Hao Luo \thanks{Corresponding author. Email: galeolev@foxmail.com},
	Xiaoping Xie \thanks{Email: xpxie@scu.edu.cn} \\
	{School of Mathematics, Sichuan University, Chengdu 610064, China}
}

\date{}
\maketitle

\begin{abstract}
  This paper analyzes a space-time finite element method for fractional wave
  problems.  The method uses a Petrov-Galerkin type time-stepping scheme to discretize
  the time fractional derivative of order $ \gamma $ ($1<\gamma<2$). We
  establish the stability of this method, and derive  the optimal
  convergence in the $ H^1(0,T;L^2(\Omega)) $-norm and suboptimal convergence in
  the discrete $ L^\infty(0,T;H_0^1(\Omega)) $-norm. Furthermore, we discuss the
  performance of this method in the case that the solution has singularity at $
  t= 0 $, and show that optimal convergence rate with respect to the $
  H^1(0,T;L^2(\Omega)) $-norm can still be achieved by using   graded grids in
  the time discretization. Finally, numerical experiments are performed to
  verify the theoretical results.


\end{abstract}


\section{Introduction}
This paper considers the following fractional wave problem:
\begin{equation}
  \label{eq:model}
  \left\{
    \begin{aligned}
      \D_{0+}^\gamma (u-u_0-tu_1) - \Delta u &= f &&
      \text{in $ \Omega \times (0,T) $,} \\
      u &= 0 &&
      \text{on $ \partial\Omega \times (0,T) $,} \\
      u(\cdot,0) &= u_0 &&
      \text{in $ \Omega $,} \\
      u_t(\cdot,0) &= u_1 &&
      \text{in $ \Omega $,}
    \end{aligned}
  \right.
\end{equation}
where $ 1 < \gamma < 2 $, $ \Omega \subset \mathbb R^d $ ($d=2,3$) is a
polygon/polyhedron, and $ u_0 $, $ u_1 $ and $ f $ are given functions. Here $
u_t $ is the derivative of $ u $ with respect to the time variable $ t $, and $
\D_{0+}^\gamma $ is a Riemann-Liouville fractional differential operator of
order $ \gamma $.

In the last two decades, the numerical treatment to time fractional
diffusion-wave partial differential equations has been an active research area.
The main difference of these numerical methods is how to discretize the fractional
derivatives. So far, there are three approaches to discretize the fractional
derivatives: the finite difference method, the spectral method, and the finite
element method. For the first class of algorithms that use the finite difference
method to discretize the fractional derivatives, we refer the reader to
\cite{Meerschaert2004Finite,Yuste2005,Yuste2006,sun2006fully,Liu2011,Gao2011,Cao2013A,Huang2013Two,Wang2014Compact,Ren2017}
and the references therein. These algorithms are easy to implement, but   are
generally of low temporal accuracy. For the second class of algorithms that use the
spectral method to discretize the fractional derivatives, we refer the reader to
\cite{Li2009,Zayernouri2014Fractional,Zayernouri2012Karniadakis,Zayernouri2014Exponentially,Zheng2015,yang2016spectral,Li2017A}.
These algorithms have high-order accuracy if the solution is sufficiently
regular. Since singularity is an important feature of time fractional
diffusion-wave problems, the high-order accuracy of these algorithms is limited.
Besides, the algorithms often lead to large scale dense systems to solve. For the third
class of algorithms that use the finite element method to discretize fractional
derivatives, we refer the reader to
\cite{Mustapha2009Discontinuous,Mustapha2012Uniform,Mustapha2012Superconvergence,Mclean2015Time,mustapha2014well-posedness,Mustapha2014A,Mustapha2015Time}.
Similar to the first class of algorithms, the discrete systems arising from these
algorithms are solved successively in the time direction. Furthermore, these
algorithms possess high-order accuracy, and if the solution has singularity,
these algorithms can also have high-order accuracy by using graded grids in the
time discretization.

Due to the nonlocal property of fractional derivatives, the history information
has to be stored to compute the solution at each stage
\cite{Zhao2014Short,Ford2001The}. Hence the storage and computing cost to solve
a time fractional wave problem is significantly more expensive than that to
solve a standard wave problem. A natural idea is to develop high-order temporal
accuracy algorithms. However, it is well known that time fractional wave
problems generally have singularity at $ t=0 $, despite how regular the initial
  and boundary data are. This makes developing high-order accuracy
algorithms more challenging. As mentioned earlier, the finite difference methods
generally only have low temporal accuracy. Besides, the high-order accuracy of
the spectral method is limited by the singularity of the time fractional wave
problems. This motivates us to develop high-order accuracy methods that can also
tackle the singularity at $ t=0 $.

In this paper, we propose a space-time finite element method for the fractional wave
problem \eqref{eq:model}. This method employs a Petrov-Galerkin type time-stepping scheme to
discretize the fractional derivative, which uses   continuous piecewise
polynomials (of degree $\leqslant m$) as trial functions and   totally
discontinuous piecewise polynomials (of degree $\leqslant m-1 $) as test functions. 
We
establish the stability of this method and derive two a priori error estimates
under a reasonable regularity assumption on the solution. The estimates 
show that the proposed method possesses temporal accuracy order $ m $ in the $
H^1(0,T;L^2(\Omega)) $-norm and temporal accuracy order $ m-1/2 $
($m\geqslant2$) in the discrete $ L^\infty(0,T;H_0^1(\Omega)) $-norm, provided
the solution is sufficiently regular. Furthermore, we use the two estimates to
analyze the convergence rates of this method in the case that the solution has
singularity at $ t=0 $, indicating that using suitable graded grids in the time
discretization can still achieve temporal accuracy order $ m $ in the $
H^1(0,T;L^2(\Omega)) $-norm. Finally, we note that our analysis is quite
different from that of the aforementioned third class of algorithms with    the finite element discretization for the fractional
derivatives, and
the techniques developed in this paper can also be used to analyze other time
fractional diffusion-wave problems.


The rest of this paper is organized as follows. \cref{sec:pre} introduces some
vector valued spaces, the Riemann-Liouville fractional calculus operators, and
the weak form to problem \cref{eq:model}. \cref{sec:discr} describes a space-time finite element method, and \cref{sec:main} investigates its stability and
convergence. \cref{sec:numer} performs some numerical experiments
to verify the theoretical results.


\section{Preliminaries}
\label{sec:pre}
We first introduce some vector valued spaces. Let $ X $ be a separable Hilbert
space with an inner product $ (\cdot,\cdot)_X $ and an orthonormal basis $ \{
  e_j:\ j \in \mathbb N \} $, and let $ \mathcal O = (a,b) $ be an interval. For
  $ 0 < \alpha < \infty $, define
\[
  H^\alpha(\mathcal O;X) := \left\{
    v \in L^2(\mathcal O;X):\
    \sum_{j=0}^\infty \nm{(v,e_j)_X}_{H^\alpha(\mathcal O)}^2 < \infty
  \right\}
\]
and endow this space with the norm
\[
  \nm{\cdot}_{H^\alpha(\mathcal O; X)} := \left(
    \sum_{j=0}^\infty \nm{(\cdot,e_j)_X}_{H^\alpha(\mathcal O)}^2
  \right)^{1/2},
\]
where $ L^2(\mathcal O;X) $ is an $ X $-valued Bochner $L^2$ space. If $ 0 <
\alpha < 1/2 $, we also introduce the following two norms:
\begin{align*}
  \snm{v}_{H^\alpha(\mathcal O;X)} &:= \left(
    \sum_{j=0}^\infty \snm{(v,e_j)_X}_{H^\alpha(\mathcal O)}^2
  \right)^{1/2}, \\
  \ssnm{v}_{H^\alpha(\mathcal O;X)} &:=
  \inf_{
    \substack{
      \widetilde v \in H^\alpha((-\infty,b);X) \\
      \widetilde v|_{\mathcal O} = v
    }
  } \snm{\widetilde v}_{H^\alpha((-\infty,b);X)},
\end{align*}
for all $ v \in H^\alpha(\mathcal O;X) $. Here, $ H^\alpha(\mathcal O) $ is a
standard Sobolev space (see \cite{Tartar2007}), and
\[
  \snm{v}_{H^\alpha(\mathcal O)} := \left(
    \int_\mathbb R \snm{\xi}^{2\alpha}
    \snm{\mathcal F(v\chi_{\mathcal O})(\xi)}^2 \,\mathrm{d}\xi
  \right)^{1/2}
\]
for each $ v \in H^\alpha(\mathcal O) $ with $ 0 < \alpha < 1/2 $, where $
\mathcal F: L^2(\mathbb R) \to L^2(\mathbb R) $ is the Fourier transform
operator and $ \chi_{\mathcal O} $ is the indicator function of the interval $
\mathcal O $. Moreover, for $ 0 < \alpha < 1/2 $, we use $ H^{-\alpha}(\mathcal
O;X) $ to denote the dual space of $ H^{\alpha}(\mathcal O; X) $, where $
H^\alpha(\mathcal O;X) $ is endowed with the norm $
\snm{\cdot}_{H^{\gamma_0}(\mathcal O;X)} $. For $ v \in H^i(\mathcal O;X) $ with
$ i \in \mathbb N_{>0} $, we use $ v^{(i)} $ to denote its $ i $th weak
derivative, and $ v^{(1)} $ and $ v^{(2)} $ are abbreviated to $ v' $ and $ v''
$, respectively.

Additionally, for $ 0 \leqslant \delta < 1 $, define
\[
  L_\delta^2(\mathcal O;X) := \left\{
    v \in L^1(\mathcal O;X) :\
    \nm{v}_{L_\delta^2(\mathcal O;X)} < \infty
  \right\},
\]
where
\[
  \nm{v}_{L_\delta^2(\mathcal O;X)} := \left(
    \int_\mathcal O \snm{t}^\delta \nm{v(t)}_X^2 \, \mathrm{d}t
  \right)^{1/2}.
\]
Conventionally, $ C(\mathcal O;X) $ is the set of all $ X $-valued continuous
functions defined on $ \mathcal O $, and $ P_j(\mathcal O;X) $ is the set of all
$ X $-valued polynomials defined on $ \mathcal O $ of degree $ \leqslant j $.
For convenience, $ \nm{\cdot}_{L_\delta^2(\mathcal O;\mathbb R)} $ and $
P_j(\mathcal O;\mathbb R) $ are abbreviated to $ \nm{\cdot}_{L_\delta^2(\mathcal
O)} $ and $ P_j(\mathcal O) $, respectively.

Now we introduce the Riemann-Liouville fractional calculus operators. Let $ X $
be a Banach space and let $ -\infty \leqslant a < b \leqslant \infty $.
\begin{Def}
  For $ 0 < \alpha < \infty $, define
  \begin{align*}
    \left(\I_{a+}^{\alpha,X} v\right)(t) &:=
    \frac1{ \Gamma(\alpha) }
    \int_a^t (t-s)^{\alpha-1} v(s) \, \mathrm{d}s, \quad t\in(a,b), \\
    \left(\I_{b-}^{\alpha,X} v\right)(t) &:=
    \frac1{ \Gamma(\alpha) }
    \int_t^b (s-t)^{\alpha-1} v(s) \, \mathrm{d}s, \quad t\in(a,b),
  \end{align*}
  for all $ v \in L^1(a,b;X) $, where $ \Gamma(\cdot) $ is the gamma function.
\end{Def}

\begin{Def}
  For $ j-1 < \alpha < j $ with $ j \in \mathbb N_{>0} $, define
  \begin{align*}
    \D_{a+}^{\alpha,X} & := \D^j \I_{a+}^{j-\alpha,X}, \\
    \D_{b-}^{\alpha,X} & := (-1)^j \D^j \I_{b-}^{j-\alpha,X},
  \end{align*}
  where $ \D $ is the first-order differential operator in the distribution
  sense.
\end{Def}
\noindent Above $ L^1(a,b;X) $ is a standard $ X $-valued Bochner $ L^1 $ space.
For convenience, we shall simply use $ \I_{a+}^\alpha $, $ \I_{b-}^\alpha $, $
\D_{a+}^\alpha $ and $ \D_{b-}^\alpha $, without indicating the underlying
Banach space $ X $.

Finally, let us define the weak solution to problem \cref{eq:model}. Throughout
this paper, we assume that $ u_0 \in H_0^1(\Omega) $, $ u_1 \in L^2(\Omega) $,
and $ f \in H^{-\gamma_0}(0,T;L^2(\Omega)) $, where $ \gamma_0 = (\gamma-1)/2 $
and $ H^{-\gamma_0}(0,T;L^2(\Omega)) $. We call
\[
  u \in H^{1+\gamma_0}(0,T;L^2(\Omega)) \cap L^2(0,T;H_0^1(\Omega))
\]
a weak solution to problem \cref{eq:model}, if $ u(0) = u_0 $ and
\begin{equation}
  \label{eq:weak_sol}
  \begin{aligned}
    & \dual{\D_{0+}^\gamma(u-u_0-tu_1), v}_{H^{\gamma_0}(0,T;L^2(\Omega))} +
    \dual{\nabla u, \nabla v}_{\Omega \times (0,T)} \\
    =& \dual{f,v}_{H^{\gamma_0}(0,T;L^2(\Omega))}
  \end{aligned}
\end{equation}
for all $ v \in H^{\gamma_0}(0,T;L^2(\Omega)) \cap L^2(0,T;H_0^1(\Omega)) $.
Above and throughout, if $ D $ is a Lebesgue measurable set of $ \mathbb R^l $
($ l= 1,2,3,4 $) then the symbol $ \dual{p,q}_D $ means $ \int_D pq $, and if $
X $ is a Banach space then $ \dual{\cdot,\cdot}_X $ means the duality pairing
between $ X^* $ (the dual space of $ X $) and $ X $.



\section{Discretization}
\label{sec:discr}
For $ \sigma \geqslant 1 $ and $ J \in \mathbb N_{>0} $, define
\[
  t_j :=(j/J)^\sigma T \quad\text{ for all } 0 \leqslant j \leqslant J,
\]
and we use $ \tau $ to abbreviate $ \tau_J $. when For each $ 1 \leqslant j \leqslant
J $, set $ \tau_j := t_j-t_{j-1} $ and $ I_j := (t_{j-1},t_j) $. Notice that $\bigcup\{I_j\} $ is   an equidistributed grid if  $ \sigma =1 $ and a graded grid  if  $ \sigma >1 $.  Let $ \mathcal
K_h $ be a triangulation of $ \Omega $ consisting of $ d $-simplexes, and we use
$ h $ to denote the maximum diameter of the elements in $ \mathcal K_h $. Define
\begin{align*}
  S_h&:= \left\{
    v_h \in H_0^1(\Omega):\
    v_h|_K \in P_n(K),\ \forall \,K \in \mathcal K_h
  \right\},\\
  \Sigma_h &:= \left\{
    v_h \in H^1(\Omega):\
    v_h|_K \in P_n(K),\ \forall \,K \in \mathcal K_h
  \right\}, \\
  M_{h,\tau} &:= \left\{
    V \in H^1(0,T;S_h):\
    V|_{I_j} \in P_m(I_j;S_h),\ \forall\,1\leqslant j\leqslant J
  \right\}, \\
  W_{h,\tau} &:= \left\{
    V \in L^2(0,T;S_h):\
    V|_{I_j} \in P_{m-1}(I_j;S_h),\ \forall\,1\leqslant j\leqslant J
  \right\},
\end{align*}
where $ m, n \in \mathbb N_{>0} $.

Now, inspired by the weak formulation \cref{eq:weak_sol}, we construct a
space time finite element method as follows: seek $ U \in M_{h,\tau} $ such that
$ U(0) = R_h u_0 $ and
\begin{equation}
  \label{eq:algo}
  \begin{aligned}
    & \dual{\D_{0+}^\gamma (U-U(0)-t\Pi_hu_1),V}_{H^{\gamma_0}(0,T;L^2(\Omega))} +
    \dual{\nabla U,\nabla V}_{\Omega \times (0,T)} \\
    = & \dual{f,V}_{H^{\gamma_0}(0,T;L^2(\Omega))}
  \end{aligned}
\end{equation}
for all $ V \in W_{h,\tau} $, where $ \Pi_h $ is the $ L^2(\Omega) $-orthogonal
projection operator onto $ \Sigma_{h} $, and $ R_h: H_0^1(\Omega) \to
S_h $ is defined by
\[
  \dual{\nabla (v-R_hv), \nabla v_h}_\Omega = 0,
  \quad \forall\,v \in H_0^1(\Omega), \; \forall\,v_h \in S_h.
\]

\begin{rem}
  Given $ V \in M_{h,\tau} $, a straightforward calculation yields that
  \begin{align*}
    &
    \dual{ \D_{0+}^\gamma \big( V-V(0) \big), \varphi } =
    \dual{\I_{0+}^{2-\gamma} (V-V(0)), \varphi''}_{(0,T)} =
    \dual{V-V(0), \I_{T-}^{2-\gamma} \varphi''}_{(0,T)} \\
    = &
    \dual{
      V-V(0), \left( \I_{T-}^{2-\gamma} \varphi\right )''
    }_{(0,T)} =
    -\dual{
      V', \left( \I_{T-}^{2-\gamma} \varphi\right )'
    }_{(0,T)} \\
    = &
    \sum_{j=1}^J \dual{V'', \I_{T-}^{2-\gamma} \varphi}_{I_j} +
    \sum_{j=1}^{J-1} \jmp{V_j'} \I_{T-}^{2-\gamma}\varphi(t_j)
  \end{align*}
  for any $ \varphi \in C_0^\infty(0,T) $, where $ \jmp{V_j'} := \lim_{t \to
  t_{j+}} V'(t) - \lim_{t \to t_{j-}} V'(t) $. Setting $ Z \in L^1((0,T); S_h) $
  by
  \[
    Z|_{I_j} := (V|_{I_j})'' \quad
    \text{ for all } 1 \leqslant j \leqslant J,
  \]
  we obtain
  \[
    \sum_{j=1}^J \dual{V'', \I_{T-}^{2-\gamma} \varphi}_{I_j} =
    \dual{Z, \I_{T-}^{2-\gamma} \varphi}_{(0,T)} =
    \dual{\I_{0+}^{2-\gamma} Z, \varphi}_{(0,T)}.
  \]
  Additionally, a simple computing gives
  \[
    \sum_{j=1}^{J-1} \jmp{V_j'} \I_{T-}^{2-\gamma} \varphi (t_j) =
    \dual{\sum_{j=1}^{J-1} \jmp{V_j'} \omega_j, \varphi}_{(0,T)},
  \]
  where
  \[
    \omega_j :=
    \begin{cases}
      0 & \text{ if } 0 < t < t_j, \\
      \frac{(t-t_j)^{1-\gamma}}{\Gamma(2-\gamma)}
      & \text{ if } t_j < t < T.
    \end{cases}
  \]
  Consequently,
  \[
    \dual{ \D_{0+}^\gamma \big(V-V(0)\big), \varphi } =
    \dual{
      \I_{0+}^{2-\gamma} Z + \sum_{j=1}^{J-1} \jmp{V_j'} \omega_j, \varphi
    }_{(0,T)}
  \]
  for all $ \varphi \in C_0^\infty(0,T) $, which indicates
  \[
    \D_{0+}^\gamma \big( V-V(0) \big) = \I_{0+}^{2-\gamma} Z +
    \sum_{j=1}^{J-1} \jmp{V_j'} \omega_j.
  \]
\end{rem}
\begin{rem}
  Since $ W_{h,\tau} $ is totally discontinuous, we can solve $ U|_{I_j} $
  successively from $ j = 1 $ to $ j = J $.
\end{rem}


\section{Stability and Convergence}
\label{sec:main}
For convenience, $ a \lesssim b $ means that there exists a positive constant $
C $, depending only on $ \gamma $, $ T $, $ m $, $ n $, or the shape regularity
of $ \mathcal K_h $, such that $ a \leqslant C b $, and $ a \sim b $ means $ a
\lesssim b \lesssim a $. Moreover, if the symbol $ C $ has subscript(s), then it
means a positive constant that depends only its subscript(s) unless otherwise
stated, and its value may differ at each of its occurrence(s). For example, $
C_{\gamma,T} $ is a positive constant that depends only on $ \gamma $ and $ T $,
and its value may differ at different places.

\subsection{Two Interpolation Operators}
Let $ X $ be a separable Hilbert space. For each $ v \in C((0,T];X) $, define $
P_\tau^X v $ by
\[
  \left\{
    \begin{aligned}
      & \left( P_\tau^Xv \right) \! \big|_{I_j} \in P_{m-1}(I_j;X),
      \quad \lim_{t\to t_{j}^-} \left( P_\tau^Xv \right)(t) = v(t_j), \\
      & \dual{ v- P_\tau^X v, q }_{I_j} = 0
      \quad \text{ for all } q \in P_{m-2}(I_j),
    \end{aligned}
  \right.
\]
for all $ 1\leqslant j \leqslant J $, where $ P_{-1}(I) := \emptyset $ in the
case of $ m = 1 $. For any $ v \in H^{1+\gamma_0}(0,T;X) $, define $ Q_\tau^X v
\in C([0,T];X) $ by
\[
  \left\{
    \begin{aligned}
      & (Q_\tau^X v)(0) = v(0), && \\
      & \left( Q_\tau^X v\right)|_{I_j} \in P_m(I_j;X) &&
      \text{ for all } 1\leqslant j \leqslant J,\\
      & \dual{
      \D_{0+}^{2\gamma_0} \left( v-Q_\tau^X v \right)',
      w_\tau
    }_{H^{\gamma_0}(0,T)} = 0 && \text{ for all } w_\tau \in W_\tau,
  \end{aligned}
  \right.
\]
where
\[
  W_\tau := \left\{
    w_\tau \in L^2(0,T):\
    w_\tau|_{I_j} \in P_{m-1}(I_j),\,\forall\,1\leqslant j \leqslant J
  \right\}.
\]
For simplicity, we shall suppress the superscript $ X $ of $ Q_\tau^X $ and $
P_\tau^X $ when no confusion will arise.

\begin{rem}
  Clearly, \cref{lem:coer} implies that $ Q_\tau^X $ is well-defined.
\end{rem}

\begin{lem}[\cite{Thomee1997}]
  \label{lem:ptau}
  If $ 1\leqslant j\leqslant J $ and $ v\in H^{m}(I_j) $, then
  \[
    \nm{v-P_\tau v}_{L^2(I_j)} + \tau_j\nm{v-P_\tau v}_{H^1(I_j)}
    \lesssim \tau_j^{m}\nm{v}_{H^{m}(I_j)}.
  \]
\end{lem}

\begin{lem}[\cite{Tartar2007}]
  \label{lem:equiv}
  If $ 0 < \alpha < 1/2 $, then
  \[
    C_1 \nm{v}_{H^\alpha(0,1)}
    \leqslant \snm{v}_{H^\alpha(0,1)} \leqslant
    C_2\nm{v}_{H^\alpha(0,1)}
  \]
  for all $ v \in H^\alpha(0,1) $, and
  \[
    C_3 \snm{v}_{H^\alpha(\mathbb R)} \leqslant
    \left(
      \int_\mathbb R \int_\mathbb R
      \frac{\snm{v(s)-v(t)}^2}{\snm{s-t}^{1+2\alpha}}
      \, \mathrm{d}s \, \mathrm{d}t
    \right)^\frac12 \leqslant C_4 \snm{v}_{H^\alpha(\mathbb R)}
  \]
  for all $ v \in H^\alpha(\mathbb R) $, where $ C_1 $, $ C_2 $, $ C_3 $ and $
  C_4 $ are four positive constants that depend only on $ \alpha $.
\end{lem}

\begin{lem}
  \label{lem:frac_appro}
  If $ 0 < \alpha < 1/2 $ and $ v \in H^{\alpha}(0,T) $ with $ v'\in L^1(0,T) $,
  then
  \[
    \snm{v-P_\tau v}_{H^{\alpha}(0,t_j)} \leqslant C_{\alpha,\gamma}
    \left(\sum_{i=1}^j \tau_i^{2-2\alpha} \inf_{0 \leqslant \delta < 1}
      \frac{t_i^{-\delta}}{1-\delta}
    \nm{\left(v - P_\tau v\right)'}_{L_\delta^2(I_i)}^2\right)^{\frac12}
  \]
  for each $ 1 \leqslant j \leqslant J $.
\end{lem}
\begin{proof}
  Setting $ g:= (I-P_\tau)v $, by \cref{lem:equiv} we only need to prove
  \begin{equation}
    \label{eq:I123}
    \begin{aligned}
      E_1 + E_2 + E_3
      \leqslant C_{\alpha,\gamma} \sum_{i=1}^j \tau_i^{2-2\alpha}
      \inf_{0\leqslant\delta<1} \frac{t_i^{-\delta}}{1-\delta}
      \nm{(v-P_\tau v)'}^2_{L^2_\delta(I_i)},
    \end{aligned}
  \end{equation}
  where
  \begin{align*}
    E_1 &= \sum_{i=1}^j
    \int_{t_{i-1}}^{t_i} \, \mathrm{d}t
    \int_{t_{i-1}}^{t_i} \snm{g(t)-g(s)}^2\snm{t-s}^{-1-2\alpha} \, \mathrm{d}s, \\
    E_2 &= \sum_{i=1}^j \sum_{l=i+1}^j
    \int_{t_{i-1}}^{t_i} \, \mathrm{d}t
    \int_{t_{l-1}}^{t_l} \snm{g(t)-g(s)}^2\snm{t-s}^{-1-2\alpha} \mathrm{d}s, \\
    E_3 &= \int_0^{t_j} \snm{g(t)}^2
    \left(
      \int_{t_j}^\infty (s-t)^{-1-2\alpha} \,\mathrm{d}s +
      \int_{-\infty}^0 (t-s)^{-1-2\alpha} \,\mathrm{d}s
    \right)\, \mathrm{d}t.
  \end{align*}
  Let us first observe that a straightforward calculation yields
  \begin{align*}
    {} &
    \sum_{i=1}^j \sum_{l=i+1}^j
    \int_{t_{i-1}}^{t_i} \, \mathrm{d}t \int_{t_{l-1}}^{t_l}
    g^2(t)\snm{t-s}^{-1-2\alpha} \, \mathrm{d}s \\
    ={} &
    \frac1{2\alpha} \sum_{i=1}^j \sum_{l=i+1}^j \int_{t_{i-1}}^{t_i}
    g^2(t) \left(
      (t_{l-1}-t)^{-2\alpha} - (t_l-t)^{-2\alpha}
    \right)
    \, \mathrm{d}t  \\
    \leqslant{} &
    \frac1{2\alpha} \sum_{i=1}^{j-1} \int_{t_{i-1}}^{t_i}\
    g^2(t) (t_i-t)^{-2\alpha}
    \, \mathrm{d}t
  \end{align*}
  and
  \begin{align*}
    {} &
    \sum_{i=1}^j \sum_{l=i+1}^j
    \int_{t_{i-1}}^{t_i} \, \mathrm{d}t \int_{t_{l-1}}^{t_l}
    g^2(s)\snm{t-s}^{-1-2\alpha} \, \mathrm{d}s \\
    ={} &
    \frac1{2\alpha} \sum_{i=1}^j \sum_{l=i+1}^j \int_{t_{l-1}}^{t_l} g^2(s)
    \left(
      (s-t_i)^{-2\alpha} - (s-t_{i-1})^{-2\alpha}
    \right) \, \mathrm{d}s \\
    \leqslant{} &
    \frac1{2\alpha} \sum_{l=2}^j \int_{t_{l-1}}^{t_l}
    g^2(s) (s-t_{l-1})^{-2\alpha}
    \, \mathrm{d}s.
  \end{align*}
  It follows that
  \[
    E_2 \leqslant
    \frac1\alpha \sum_{i=1}^j
    \int_{t_{i-1}}^{t_i} g^2(t) \big(
      (t_i-t)^{-2\alpha} + (t-t_{i-1})^{-2\alpha}
    \big) \, \mathrm{d}t.
  \]
  In addition, it is evident that
  \begin{align*}
    E_3 \leqslant
    \frac1{2\alpha} \sum_{i=1}^j \int_{t_{i-1}}^{t_i}
    g^2(t) \big((t_i-t)^{-2\alpha} + (t-t_{i-1})^{-2\alpha}\big) \,
    \mathrm{d}t.
  \end{align*}
  Therefore, using \cref{lem:jm} yields
  \begin{align*}
    E_2 + E_3 &
    \leqslant \frac3{2\alpha} \sum_{i=1}^j
    \int_{t_{i-1}}^{t_i} g^2(t)
    \big( (t_i-t)^{-2\alpha} + (t-t_{i-1})^{-2\alpha} \big)
    \, \mathrm{d}t \\
    & \leqslant
    C_{\alpha,\gamma} \sum_{i=1}^j \tau_i^{2-2\alpha}\inf_{0\leqslant\delta<1}
    \frac{t_i^{-\delta}}{1-\delta}\nm{g'}^2_{L^2_\delta(I_i)}.
  \end{align*}
  As \cref{lem:jm} also implies
  \[
    E_1 \leqslant C_\alpha\sum_{i=1}^j \tau_i^{2-2\alpha}
    \inf_{0\leqslant\delta<1} \frac{t_i^{-\delta}}{1-\delta}
    \nm{g'}^2_{L^2_\delta(I_i)},
  \]
  we readily obtain \cref{eq:I123} and thus conclude the proof.
\end{proof}

\begin{lem}
  \label{lem:graded}
  Define
  \[
    v(t) := t^r, \quad 0 < t < T,
  \]
  where $ 1 < r \leqslant m + 1/2 $ and $ r \not\in \mathbb N $. If $ \sigma=1
  $, then
  \begin{equation}
    \label{eq:equiv_grid}
    \sum_{i=1}^j \tau_i^{3-\gamma} \inf_{0 \leqslant \delta < 1}
    \frac{t_i^{-\delta}}{1-\delta}
    \nm{\left(v'- P_\tau v'\right)'}_{L_\delta^2(I_i)}^2
    \leqslant C_{r,m,\gamma,T} J^{-(2r-\gamma)}.
  \end{equation}
  If $ \sigma > \sigma^* $, then
  \begin{equation}
    \label{eq:graded_grid}
    \sum_{i=1}^j \tau_i^{3-\gamma} \inf_{0 \leqslant \delta < 1}
    \frac{t_i^{-\delta}}{1-\delta}
    \nm{\left(v'- P_\tau v'\right)'}_{L_\delta^2(I_i)}^2
    \leqslant C_{r,m,\gamma,\sigma,T} J^{-(2m+1-\gamma)}
  \end{equation}
  for each $ 1 \leqslant j \leqslant J $, where
  \[
    \sigma^* := \frac{2m+1-\gamma}{2r-\gamma}.
  \]
\end{lem}
\begin{proof}
  Since the proof of \cref{eq:equiv_grid} is a trivial modification of that of
  \cref{eq:graded_grid}, we only prove the latter. A standard scaling argument
  yields
  \[
    \nm{(P_\tau v')'}_{L_{\delta_0}^2(I_1)} \leqslant
    C_r \nm{v''}_{L_{\delta_0}^2(I_1)} \leqslant
    C_r \tau_1^{(2r+\delta_0-3)/2},
  \]
  where
  \[
    \delta_0 :=
    \begin{cases}
      1/2 & \text{ if } r \geqslant 3/2,\\
      2-r & \text{ if } 1 < r < 3/2.
    \end{cases}
  \]
  It follows that
  \begin{align*}
    & \tau_1^{3-\gamma} \inf_{0 \leqslant \delta < 1} \frac{t_1^{-\delta}}{1-\delta}
    \nm{(v'-P_\tau v')'}_{L_\delta^2(I_1)}^2 \\
    \leqslant{} &
    \tau_1^{3-\gamma} \frac{\tau_1^{-\delta_0}}{1-\delta_0}
    \nm{(v'-P_\tau v')'}_{L_{\delta_0}^2(I_1)}^2 \\
    \leqslant{} &
    C_r \tau_1^{2r-\gamma}.
  \end{align*}
  Therefore, by the evident estimate
  \[
    \tau_1^{2r-\gamma} \leqslant C_{r,\gamma,\sigma,T} J^{-(2m+1-\gamma)},
  \]
  we obtain
  \begin{equation}
    \label{eq:733}
    \tau_1^{3-\gamma} \inf_{0 \leqslant \delta < 1} \frac{t_1^{-\delta}}{1-\delta}
    \nm{(v'-P_\tau v')'}_{L_\delta^2(I_1)}^2
    \leqslant C_{r,\gamma,\sigma,T} J^{-(2m+1-\gamma)}.
  \end{equation}

  In addition, \cref{lem:ptau} implies
  \[
    \begin{aligned}
      &\sum_{i=2}^j \tau_i^{3-\gamma} \inf_{0 \leqslant \delta < 1}
      \frac{t_i^{-\delta}}{1-\delta} \nm{(v'-P_\tau v')'}_{L_\delta^2(I_i)}^2 \\
      \leqslant{} &
      \sum_{i=2}^j \tau_i^{3-\gamma} \nm{(v'-P_\tau v')'}_{L^2(I_i)}^2 \\
      \leqslant{} &
      C_m \sum_{i=2}^j \tau_i^{3-\gamma+2(m-1)}
      \int_{t_{i-1}}^{t_i} t^{2(r-m-1)} \, \mathrm{d}t.
    \end{aligned}
  \]
  Then, by the inequality
  \[
    \tau_i < \sigma 2^{\sigma-1} J^{-1} T^{1/\sigma} t_{i-1}^{1-1/\sigma},
    \quad 2 \leqslant i \leqslant j,
  \]
  we obtain
  \begin{equation}
    \label{eq:734}
    \begin{aligned}
      &\sum_{i=2}^j \tau_i^{3-\gamma} \inf_{0 \leqslant \delta < 1}
      \frac{t_i^{-\delta}}{1-\delta} \nm{(v'-P_\tau v')'}_{L_\delta^2(I_i)}^2 \\
      \leqslant{}&
      C_{m,\gamma,\sigma,T} J^{-(2m+1-\gamma)} \int_{t_1}^{t_j}
      t^{2(r-m-1) + (1-1/\sigma)(2m+1-\gamma)} \, \mathrm{d}t \\
      \leqslant{}& C_{r,m,\gamma,\sigma,T} J^{-(2m+1-\gamma)}.
    \end{aligned}
  \end{equation}

  Finally, combining \cref{eq:733,eq:734} yields \cref{eq:graded_grid} and thus
  proves the lemma.
\end{proof}

\begin{lem}
  \label{lem:xy}
  If $ \gamma_0 \leqslant \beta < \infty $ and $ v \in H^\alpha(0,T) $ with $ 0
  \leqslant \alpha < 1/2 $, then
  \begin{equation}
    \label{eq:xy}
    \inf_{w_\tau \in W_\tau}
    \ssnm{\I_{t_j-}^\beta v - w_\tau}_{H^{\gamma_0}(0,t_j)}
    \leqslant C_{\alpha,\beta,m,\gamma,\sigma,T} \tau_j^{\min\{\alpha+\beta-\gamma_0,m-\gamma_0\}}
    \snm{v}_{H^{\varrho}(0,t_j)}
  \end{equation}
  for each $ 1 \leqslant j \leqslant J $,
  where
  \[
    \varrho :=\min\big\{\alpha,\max\{0,m-\beta\}\big\}.
  \]
\end{lem}
\begin{proof}
  Let $ J^* $ be the smallest integer such that $ T/J^* \leqslant \tau_1 $, and
  define
  \[
    W_\tau^* := \left\{
      w_\tau \in L^2(-T,T):\
      w_\tau|_{I_i} \in P_{m-1}(I_i),\
      {-(J^*-1)} \leqslant i\leqslant J
    \right\},
  \]
  where
  \[
    I_i := \left( {\frac{i-1}{J^*}} T, {\frac i{J^*}} T \right)
    \quad \text{ for all }  {-(J^*-1)} \leqslant i \leqslant 0.
  \]
  Extending $ v $ to $ (-T,0) $ by zero, by the definition of the norm $
  \ssnm{\cdot}_{H^{\gamma_0}(0,t_j)} $ we obtain
  \begin{align*}
    \inf_{w_\tau \in W_\tau}
    \ssnm{\I_{t_j-}^\beta v - w_\tau}_{H^{\gamma_0}(0,t_j)}
    & \leqslant
    \inf_{w_\tau \in W_\tau^*}
    \ssnm{
      \I_{t_j-}^\beta v - w_\tau
    }_{H^{\gamma_0}(-T,t_j)} \\
    & \leqslant
    \inf_{w_\tau \in W_\tau^*}
    \snm{\I_{t_j-}^\beta v - w_\tau}_{H^{\gamma_0}(-T,t_j)},
  \end{align*}
  so that
  \begin{equation*}
    \inf_{w_\tau \in W_\tau}
    \ssnm{\I_{t_j-}^\beta v - w_\tau}_{H^{\gamma_0}(0,t_j)}
    \leqslant C_{\gamma,T}
    \inf_{w_\tau \in W_\tau^*}
    \nm{\I_{t_j-}^\beta v - w_\tau}_{H^{\gamma_0}(-T,t_j)},
  \end{equation*}
  by \cref{lem:equiv}. Additionally, a standard scaling argument yields
  \[
    \nm{\I_{t_j-}^\beta v}_{H^{\min\{\alpha+\beta,m\}}(-T,t_j)}
    \leqslant C_{\alpha,\beta,m,T} \snm{v}_{H^\varrho(-T,t_j)}
    = C_{\alpha,\beta,m,T} \snm{v}_{H^\varrho(0,t_j)},
  \]
  by \cref{lem:regu,lem:equiv}. Therefore, \cref{eq:xy} follows from the
  standard approximation estimate (see \cite[Chapter~14]{Brenner2008})
  \begin{align*}
    &
    \inf_{w_\tau \in W_\tau^*}
    \nm{\I_{t_j-}^\beta v - w_\tau}_{H^{\gamma_0}(-T,t_j)} \\
    \leqslant{} &
    C_{\alpha,\beta,m,\gamma,\sigma,T} \,
    \tau_j^{\min\{\alpha+\beta-\gamma_0,m-\gamma_0\}}
    \nm{\I_{t_j-}^\beta v}_{H^{\min\{\alpha+\beta,m\}}(-T,t_j)}.
  \end{align*}
  This concludes the proof of the lemma.
\end{proof}
\begin{rem}
  Observe that the constant in \cref{eq:xy} is independent of $ t_j $, which is
  crucial in our analysis.
\end{rem}

\begin{lem}
  \label{lem:I-Q}
  If $ v \in H^{1+\gamma_0}(0,T) $ and $ v' \in C(0,T] $, then
  \begin{align}
    \snm{(v-Q_\tau v)'}_{H^{\gamma_0}(0,t_j)} &
    \lesssim \snm{(I-P_\tau)v'}_{H^{\gamma_0}(0,t_j)},
    \label{eq:I-Q-g'-frac} \\
    \nm{(v-Q_\tau v)'}_{L^2(0,t_j)} &
    \lesssim \tau_j^{\gamma_0} \snm{(I-P_\tau)v'}_{H^{\gamma_0}(0,t_j)},
    \label{eq:I-Q-g'-L2} \\
    \nm{v-Q_\tau v}_{L^2(0,t_j)} &
    \lesssim \tau_j^{\min\{1+\gamma_0,m-\gamma_0\}}
    \snm{(I-P_\tau) v'}_{H^{\gamma_0}(0,t_j)},
    \label{eq:I-Q-g-L2} \\
    \nm{v-Q_\tau v}_{H^{-\gamma_0}(0,t_j)} &
    \lesssim \tau_j^{\min\{\gamma,m-\gamma_0\}}
    \snm{(I-P_\tau) v'}_{H^{\gamma_0}(0,t_j)},
    \label{eq:I-Q-g-dual}
  \end{align}
  for each $ 1 \leqslant j \leqslant J $. Moreover, if $ m \geqslant 2 $, then
  \begin{equation}
    \label{eq:I-Q-g-inf}
    \snm{(v-Q_\tau v)(t_j)} \lesssim \tau_j^{1/2+\gamma_0}
    \snm{(I-P_\tau)v'}_{H^{\gamma_0}(0,t_j)}.
  \end{equation}
\end{lem}
\begin{proof}
  Set $ g := v-Q_\tau v $ and let us first prove \cref{eq:I-Q-g'-frac}. Observing
  that the definition of $ Q_\tau $ implies
  \[
    \dual{
      \D_{0+}^{2\gamma_0}(v-Q_\tau v)',
      (Q_\tau v)'- P_\tau v'
    }_{H^{\gamma_0}(0,t_j)} = 0,
  \]
  we obtain
  \[
    \begin{aligned}
      {}&
      \dual{
          \D_{0+}^{2\gamma_0}\big( (Q_\tau v)'-P_\tau v' \big),
          (Q_\tau v)'-P_\tau v'
        }_{H^{\gamma_0}(0,t_j)} \\
      ={}&
      \dual{
          \D_{0+}^{2\gamma_0}( v'-P_\tau v'),
          (Q_\tau v)'-P_\tau v'
        }_{H^{\gamma_0}(0,t_j)}.
    \end{aligned}
  \]
  Therefore, using \cref{lem:coer} yields
  \[
    \snm{(Q_\tau v)'-P_\tau v'}_{H^{\gamma_0}(0,t_j)}
    \lesssim \snm{ (I-P_\tau) v'}_{H^{\gamma_0}(0,t_j)},
  \]
  and so \cref{eq:I-Q-g'-frac} follows from the triangle inequality
  \[
    \snm{g'}_{H^{\gamma_0}(0,t_j)} \leqslant
    \snm{(Q_\tau v)' - P_\tau v'}_{H^{\gamma_0}(0,t_j)} +
    \snm{(I-P_\tau) v'}_{H^{\gamma_0}(0,t_j)}.
  \]

  Then let us prove \cref{eq:I-Q-g'-L2}. Since \cref{lem:coer} implies
  \[
    \nm{g'}_{L^2(0,t_j)}^2 =
    \dual{g',\D^{2\gamma_0}_{t_j-}\I^{2\gamma_0}_{t_j-} g'}_{(0,t_j)} =
    \dual{
      \D_{0+}^{2\gamma_0}g',
      \I^{2\gamma_0}_{t_j-} g'
    }_{H^{\gamma_0}(0,t_j)},
  \]
  the definition of $ Q_\tau $ implies that
  \[
    \nm{g'}_{L^2(0,t_j)}^2 =
    \dual{
      \D_{0+}^{2\gamma_0} g',
      \I^{2\gamma_0}_{t_j-} g'-w_\tau
    }_{H^{\gamma_0}(0,t_j)}
  \]
  for all $ w_\tau \in W_\tau $. Therefore, by \cref{lem:43,lem:xy} we obtain
  \begin{align*}
    \nm{g'}_{L^2(0,t_j)}^2 & \lesssim
    \snm{g'}_{H^{\gamma_0}(0,t_j)} \inf_{w_\tau \in W_\tau}
    \ssnm{\I_{t_j-}^{2\gamma_0} g' -w_\tau}_{H^{\gamma_0}(0,t_j)} \\
    & \lesssim \snm{g'}_{H^{\gamma_0}(0,t_j)}
    \tau_j^{\gamma_0} \nm{g'}_{L^2(0,t_j)}.
  \end{align*}
  It follows that
  \[
    \nm{g'}_{L^2(0,t_j)} \lesssim \tau_j^{\gamma_0}
    \snm{g'}_{H^{\gamma_0}(0,t_j)},
  \]
  which, together with \cref{eq:I-Q-g'-frac}, proves estimate
  \cref{eq:I-Q-g'-L2}.

  Analogously, we can obtain \cref{eq:I-Q-g-L2,eq:I-Q-g-dual}. Since $ g(0) = 0
  $, using integration by parts gives
  \[
    \snm{g(t_j)}^2 = \int_0^{t_j} 2g(t)g'(t) \, \mathrm{d}t
    \leqslant 2 \nm{g}_{L^2(0,t_j)} \nm{g'}_{L^2(0,t_j)}.
  \]
  Therefore, combining \cref{eq:I-Q-g'-L2,eq:I-Q-g-L2} proves
  \cref{eq:I-Q-g-inf}. This completes the proof.
\end{proof}

\begin{lem}
  \label{lem:I-Q-inf-1}
  If $ m=1 $ and $ v \in H^{1+\gamma_0}(0,T) $, then
  \begin{equation}
    \label{eq:I-Q-inf-1}
    \frac{\snm{(v-Q_\tau v)(t_j)}}
    {\snm{(v-Q_\tau v)'}_{H^{\gamma_0}(0,t_j)}}
    \lesssim\left\{
      \begin{aligned}
        &\tau_j^{\gamma_0+1/2}
        & \text{ if }  \gamma_0 < 1/4, \\
        &\sqrt{1+\ln(t_j/\tau_j)} \, \tau_j^{1-\gamma_0}
        & \text{ if } \gamma_0 = 1/4, \\
        &\tau_j^{1-\gamma_0}
        & \text{ if } \gamma_0 > 1/4,
      \end{aligned}
    \right.
  \end{equation}
  for each $ 1 \leqslant j \leqslant J $.
\end{lem}
\begin{proof}
  Putting
  \[
    G(t) := \frac{(t_j-t)^{2\gamma_0}}{\Gamma(1-2\gamma_0)},
    \quad 0 < t < t_j,
  \]
  by a direct computing we obtain
  \begin{align*}
    (v-Q_\tau v)(t_j) =
    \dual{\D_{0+}^{2\gamma_0} (v-Q_\tau v)', G}_{H^{\gamma_0}(0,t_j)}.
  \end{align*}
  From the definition of $ Q_\tau $ it follows that
  \begin{align*}
    (v-Q_\tau v)(t_j) =
    \dual{
      \D_{0+}^{2\gamma_0} (v-Q_\tau v)', G-w_\tau
    }_{H^{\gamma_0}(0,t_j)}
  \end{align*}
  for all $ w_\tau \in W_\tau $. Hence \cref{lem:coer} implies
  \[
    \snm{(v-Q_\tau v)(t_j)} \lesssim
    \snm{(v-Q_\tau v)'}_{H^{\gamma_0}(0,t_j)}
    \inf_{w_\tau \in W_\tau} \snm{G-w_\tau}_{H^{\gamma_0}(0,t_j)}.
  \]
  By the similar techniques as that used in \cref{lem:frac_appro,lem:graded}, a
  tedious but straightforward calculation yields
  \[
    \inf_{w_\tau \in W_\tau} \snm{G-w_\tau}_{H^{\gamma_0}(0,t_j)} \lesssim
    \begin{cases}
      \tau_j^{\gamma_0+1/2} &
      \text{ if }  \gamma_0 < 1/4, \\
      \sqrt{1+\ln(t_j/\tau_j)} \, \tau_j^{1-\gamma_0} &
      \text{ if } \gamma_0 = 1/4, \\
      t_j^{2\gamma_0-1/2} \tau_j^{1-\gamma_0} &
      \text{ if } \gamma_0 > 1/4.
    \end{cases}
  \]
  Combining the above two estimates gives \cref{eq:I-Q-inf-1} and thus concludes
  the proof.
\end{proof}

\subsection{Main Results}
In the rest of this paper, we assume that $ u $ and $ U $ are the solutions to
problem \cref{eq:weak_sol} and \cref{eq:algo}, respectively. Moreover, for each
$ 1 \leqslant j \leqslant J $ we define $ \epsilon_j $ as follows: if $ m = 1 $
then set
\[
  \epsilon_j :=
  \begin{cases}
    \tau_j^{1/2+\gamma_0} & \text{ if } \gamma_0 < 1/4, \\
    \sqrt{1+\ln(t_j/\tau_j)} \, \tau_j^{1-\gamma_0} & \text{ if } \gamma_0 = 1/4, \\
    \tau_j^{1-\gamma_0} & \text{ if } \gamma_0 > 1/4,
  \end{cases}
\]
and if $ m \geqslant 2 $ then set
\[
  \epsilon_j := \tau_j^{1/2+\gamma_0}.
\]
\begin{thm}
  \label{thm:stabi}
  It holds that
  \begin{equation}
    \label{eq:stabi}
    \begin{aligned}
      & \snm{U'}_{H^{{\gamma_0}}(0,t_j;L^2(\Omega))} +
      \nm{U(t_j)}_{H_0^1(\Omega)} \\
      \lesssim & \nm{ u_0}_{H_0^1(\Omega)} +
      t_j^{1/2-\gamma_0}\nm{u_1}_{L^2(\Omega)} +
      \nm{f}_{H^{-\gamma_0}(0,t_j;L^2(\Omega))}
    \end{aligned}
  \end{equation}
  for each $ 1 \leqslant j \leqslant J $.
\end{thm}
\begin{rem}
  Due to the linearity of \cref{eq:algo}, the above theorem also implies the
  unique existence of $ U $.
\end{rem}
\begin{rem}
  We recall that $ H^{-\gamma_0}(0,t_j;L^2(\Omega)) $ is the dual space of $
  H^{\gamma_0}(0,t_j;L^2(\Omega)) $, where $ H^{\gamma_0}(0,t_j;L^2(\Omega)) $
  is endowed with the norm $ \snm{\cdot}_{H^{\gamma_0}(0,t_j;L^2(\Omega))} $.
  Using the same technique as that used in the proof of \cref{lem:73}, we easily
  derive that
  \[
    L_{2\gamma_0}^2(0,t_j;L^2(\Omega)) \subset H^{-\gamma_0}(0,t_j;L^2(\Omega)).
  \]
  This indicates that even if $ f $ has singularity at $ t=0 $, problem
  \cref{eq:algo} may also be stable. Moreover, since
  \[
    \snm{v}_{H^{\gamma_0}(0,t_j;L^2(\Omega))} =
    \snm{v}_{H^{\gamma_0}(0,T;L^2(\Omega))}
  \]
  for all $ v \in H^{\gamma_0}(0,T;L^2(\Omega)) $ such that $ v|_{(t_j,T)} = 0
  $, we obtain
  \[
    \nm{f}_{H^{-\gamma_0}(0,t_j;L^2(\Omega))} \leqslant
    \nm{f}_{H^{-\gamma_0}(0,T;L^2(\Omega))}.
  \]
\end{rem}

\begin{thm}
  \label{thm:conv}
  If $ u \in H^{1+\gamma_0}(0,T; H_0^1(\Omega)\cap H^2(\Omega)) $ and $ u'' \in
  L^1(0,T;H^2(\Omega)) $, then
  \begin{align}
    \nm{(u-U)(t_j)}_{H_0^1(\Omega)} &\lesssim
    \eta_{j,1} + \eta_{j,2} + \eta_{j,3} + \eta_{j,5},
    \label{eq:conv_1} \\
    \nm{(u-U)'}_{L^2(0,t_j;L^2(\Omega))} &\lesssim
    \eta_{j,4} + t_j^{\gamma_0}(\eta_{j,2}+\eta_{j,5}),
    \label{eq:conv_2}
  \end{align}
  for each $ 1 \leqslant j \leqslant J $, where
  \begin{align*}
    \eta_{j,1} &:= \nm{(I-R_h)u(t_j)}_{H_0^1(\Omega)}, \\
    \eta_{j,2} &:= \snm{(I-R_h)u'}_{
      H^{\gamma_0}\left( 0,t_j; L^2(\Omega) \right)
    },\\
    \eta_{j,3} &:= \epsilon_j
    \left(
      \sum_{i=1}^j \tau_i^{2-2\gamma_0} \inf_{0 \leqslant \delta < 1}
      \frac{t_i^{-\delta}}{1-\delta}
      \nm{\big( (I-P_\tau)R_hu' \big)'}_{L_\delta^2(0,t_j;H_0^1(\Omega))}
    \right)^{1/2}, \\
    \eta_{j,4} &:=
    \tau_j^{\gamma_0}
    \left(
      \sum_{i=1}^j \tau_i^{2-2\gamma_0}
      \inf_{0 \leqslant \delta < 1}
      \frac{t_i^{-\delta}}{1-\delta}
      \nm{\big((I-P_\tau)R_hu'\big)'}_{L_\delta^2(I_i,L^2(\Omega))}
    \right)^{1/2}, \\
    \eta_{j,5} &:=
    \tau_j^{\min\{\gamma,m-\gamma_0\}}
    \left(
      \sum_{i=1}^j \tau_i^{2-2\gamma_0} \inf_{0 \leqslant \delta < 1}
      \frac{t_i^{-\delta}}{1-\delta}
      \nm{\big((I-P_\tau)\Delta u'\big)'}_{L_\delta^2(I_i;L^2(\Omega))}^2
    \right)^{1/2}.
  \end{align*}
\end{thm}

\subsubsection{High Regularity Case}
By \cref{thm:conv}, \cref{lem:ptau} and the standard estimate that
(\cite{Ciarlet2002})
\[
  \nm{(I-R_h)v}_{L^2(\Omega)} + h \nm{(I-R_h)v}_{H_0^1(\Omega)}
  \lesssim h^{n+1} \nm{v}_{H^{n+1}(\Omega)}
\]
for all $ v \in H_0^1(\Omega) \cap H^{n+1}(\Omega) $, we readily conclude the
following convergence estimates.
\begin{thm}
  \label{thm:conv_regu}
  If
  \[
    u \in H^{m+1}( 0,T; H_0^1(\Omega)\cap H^2(\Omega )) \cap
    H^{1+\gamma_0}(0,T;H_0^1(\Omega)\cap H^{n+1}(\Omega)),
  \]
  then
  \begin{equation*}
    \nm{(u-U)(t_j)}_{H_0^1(\Omega)} \lesssim
    \nu_{j,1} + \nu_{j,3} + \nu_{j,4} + \nu_{j,5}
  \end{equation*}
  \begin{equation}
    \nm{(u-U)'}_{L^2(0,t_j;L^2(\Omega))} \lesssim
    \nu_{j,2} + t_j^{\gamma_0}(\nu_{j,3} + \nu_{j,5})
  \end{equation}
  for each $ 1 \leqslant j \leqslant J $, where
  \begin{align*}
    \nu_{j,1} &:=  h^n\nm{u(t_j)}_{H^{n+1}(\Omega)}, \\
    \nu_{j,2} &:= \tau_j^m \nm{u}_{H^{m+1}(0,t_j,H_0^1(\Omega))}, \\
    \nu_{j,3} &:= h^{n+1} \snm{u'}_{H^{\gamma_0}(0,t_j;H^{n+1}(\Omega))}, \\
    \nu_{j,4} &:= \epsilon_j \tau_j^{m-\gamma_0}
    \nm{u}_{H^{m+1}(0,t_j;H_0^1(\Omega))}, \\
    \nu_{j,5} &:=  \tau_j^{
      \min\{\gamma,m-\gamma_0\}-\gamma_0 + m
    } \nm{u}_{H^{m+1}(0,t_j;H^2(\Omega))}.
  \end{align*}
\end{thm}

\begin{rem}\label{rem:conv_regu}
  Assume that $ u $ is sufficiently regular. The above theorem indicates the
  following results. If $ m=1 $ then
  \[
    \max_{1 \leqslant j \leqslant J}
    \nm{(u-U)(t_j)}_{H_0^1(\Omega)} = \mathcal O(h^n) +
    \begin{cases}
      \mathcal O\left(J^{-3/2}\right) & \text{ if } \gamma_0 < 1/4, \\
      \mathcal O\left((\ln{}J)^{1/2} J^{-3/2}\right) & \text{ if } \gamma_0 = 1/4, \\
      \mathcal O\left(J^{-2(1-\gamma_0)}\right) & \text{ if } \gamma_0 > 1/4.
    \end{cases}
  \]
  If $ m \geqslant 2 $ then
  \[
    \max_{1 \leqslant j \leqslant J}
    \nm{(u-U)(t_j)}_{H_0^1(\Omega)} =
    \mathcal O\left( h^n \right) + \mathcal O\left( J^{-m-1/2} \right).
  \]
  Moreover,
  \[
    \nm{(u-U)'}_{L^2(0,T;L^2(\Omega))} =
    \mathcal O\left( h^{n+1} \right) +
    \mathcal O\left( J^{-m} \right)
  \]
  for all $ m \in \mathbb N_{>0} $.
\end{rem}

\subsubsection{Singularity Case}
Let us first consider the following fractional ordinary problem:
\begin{equation}
  \left\{
    \begin{aligned}
      & \D_{0+}^\gamma (y-c_0-tc_1) + \lambda y = g \text{ in } (0,T), \\
      & y(0) = c_0, \quad y'(0) = c_1,
    \end{aligned}
  \right.
\end{equation}
where $ c_0,c_1 \in \mathbb R $, $ \lambda \in \mathbb R_{>0} $, and $ g $ is a
given function. It is well known that we can turn the above problem into the
following integral form:
\[
  y(t) = c_0 + c_1t + \frac1{\Gamma(\gamma)}
  \int_0^t (t-s)^{\gamma-1} \big( g(s) -\lambda y(s) \big) \, \mathrm{d}s,
  \quad 0 < t < T.
\]
Suppose that $ g $ is sufficiently smooth on $ [0,T] $. It is clear that if $
g(0) \neq \lambda c_0 $, then $ y $ is dominated by
\[
  \frac{g(0) - \lambda c_0}{\Gamma(1+\gamma)} t^\gamma
\]
near $ t = 0 $. This motivates us to investigate the accuracy of $ U $ in the
case that $ u $ is of the form
\begin{equation}
  \label{eq:singu_u}
  u(x,t) = t^r \phi(x), \quad (x,t) \in \Omega \times (0,T),
\end{equation}
where $ \phi \in H_0^1(\Omega) \cap H^{n+1}(\Omega) $ and $ 1 < r \leqslant m +
1/2 $ with $ r \not\in \mathbb N $.

To this end, let us introduce $ \varepsilon_j $ for each $ 1 \leqslant j
\leqslant J $ as follows: if $ m=1 $ then define
\[
  \varepsilon_j :=
  \begin{cases}
    J^{-(1/2+\gamma_0)} & \text{ if } \gamma_0 < 1/4, \\
    \sqrt{1+\ln(t_j/\tau_j)} J^{-(1-\gamma_0)} & \text{ if } \gamma_0 = 1/4, \\
    J^{-(1-\gamma_0)} & \text{ if } \gamma_0 > 1/4,
  \end{cases}
\]
and if $ m \geqslant 2 $ then define
\[
  \varepsilon_j := J^{-(1/2+\gamma_0)}.
\]
We also set
\begin{equation}
  \label{eq:sigma*} \sigma^* := \frac{2m+1-\gamma}{2r-\gamma}.
\end{equation}
By \cref{thm:conv,lem:graded}, we easily obtain the following convergence
estimates.
\begin{thm}
  \label{thm:singular}
  If $ \sigma = 1 $, then
  \begin{align}
    \nm{(u-U)'}_{L^2(0,T;L^2(\Omega))} &\leqslant
    C_1 \left(  h^{n+1} + J^{-(r-1/2)} \right), \\
    \max_{1 \leqslant j \leqslant J} \nm{(u-U)(t_j)}_{H_0^1(\Omega)} &\leqslant
    C_2 (h^n + \varepsilon_j J^{-(r-1/2-\gamma_0)}).
  \end{align}
  Furthermore, if $ \sigma > \sigma^* $, then
  \begin{align}
    \nm{(u-U)'}_{L^2(0,T;L^2(\Omega))} &\leqslant
    C_3 \left( h^{n+1} + J^{-m}  \right), \\
    \max_{1 \leqslant j \leqslant J} \nm{(u-U)(t_j)}_{H_0^1(\Omega)} &\leqslant
    C_4( h^n + \varepsilon_j J^{-(m-\gamma_0)}).
  \end{align}
  Above $ C_1 $, $ C_2 $, $ C_3 $ and $ C_4 $ are four positive constants that
  depend only on $ m $, $ n $, $ \gamma $, $ \sigma $, $ r $, $ \phi $, $ T $
  and the regularity of $ \mathcal K_h $.
\end{thm}

\subsection{Proofs of \cref{thm:stabi,thm:conv}}
\label{ssec:proofs}
{\bf Proof of \cref{thm:stabi}.}
  Inserting $ V = U' \chi_{(0,t_j)} $ into \cref{eq:algo} yields
  \[
    \begin{aligned}
      {}&
      \dual{
        \D_{0+}^\gamma \big( U-U(0)-t\Pi_hu_1 \big), U'
      }_{H^{\gamma_0}(0,t_j;L^2(\Omega))} +
      \dual{\nabla U,\nabla U'}_{\Omega\times(0,t_j)}  \\
      ={}&
      \dual{f,U'}_{H^{\gamma_0}(0,t_j;L^2(\Omega))} +
      \dual{
        \D_{0+}^\gamma \Pi_hu_1, U'
      }_{H^{\gamma_0}(0,t_j;L^2(\Omega))}.
    \end{aligned}
  \]
  Since
  \[
    \D_{0+}^\gamma (U-U(0)-t\Pi_hu_1) =
    \D_{0+}^{2\gamma_0}(U'-\Pi_hu_1),
  \]
  it follows that
  \[
    \begin{aligned}
      {}&
      \dual{
        \D_{0+}^{2\gamma_0} \big( U'\Pi_hu_1 \big), U'
      }_{H^{\gamma_0}(0,t_j;L^2(\Omega))} +
      \dual{\nabla U,\nabla U'}_{\Omega\times(0,t_j)}  \\
      ={}&
      \dual{f,U'}_{H^{\gamma_0}(0,t_j;L^2(\Omega))} +
      \dual{
        \D_{0+}^\gamma \Pi_hu_1, U'
      }_{H^{\gamma_0}(0,t_j;L^2(\Omega))}.
    \end{aligned}
  \]
  In addition, using integration by parts gives
  \[
      2 \dual{ \nabla U, \nabla U' }_{\Omega\times(0,t_j)} =
      \nm{U(t_j)}_{H_0^1(\Omega)}^2 -
      \nm{U(0)}_{H_0^1(\Omega)}^2,
  \]
  and \cref{lem:coer} implies
  \[
    \begin{aligned}
      \dual{\D_{0+}^{2\gamma_0} U', U'}_{H^{\gamma_0}(0,t_j;L^2(\Omega))} &
      \sim  \snm{U'}_{H^{\gamma_0}(0,t_j;L^2(\Omega))}^2 ,\\
      \dual{\D_{0+}^{2\gamma_0} \Pi_hu_1, U'}_{H^{\gamma_0}(0,t_j;L^2(\Omega))} &
      \lesssim \snm{\Pi_hu_1}_{H^{\gamma_0}(0,t_j;L^2(\Omega))}
      \snm{U'}_{H^{\gamma_0}(0,t_j;L^2(\Omega))}.
    \end{aligned}
  \]
  Consequently,
  \[
    \begin{aligned}
      {}& \snm{U'}_{H^{\gamma_0}(0,t_j;L^2(\Omega))}^2 +
      \nm{U(t_j)}_{H_0^1(\Omega)}^2 \\
      \lesssim{}& \nm{U(0)}_{H_0^1(\Omega)}^2 +
      \dual{f,U'}_{H^{\gamma_0}(0,t_j;L^2(\Omega))} +
      \snm{\Pi_hu_1}_{H^{\gamma_0}(0,t_j;L^2(\Omega))}
      \snm{U'}_{H^{\gamma_0}(0,t_j;L^2(\Omega))} \\
      \lesssim{}& \nm{U(0)}_{H_0^1(\Omega)}^2 +
      \left(
        \nm{f}_{H^{-\gamma_0}(0,T_j;L^2(\Omega))} +
        \snm{\Pi_hu_1}_{H^{\gamma_0}(0,t_j;L^2(\Omega))}
      \right)
      \snm{U'}_{H^{\gamma_0}(0,t_j;L^2(\Omega))}.
    \end{aligned}
  \]
  By the Young's inequality with $ \epsilon $, it follows that
  \[
    \begin{aligned}
      & \snm{U'}_{H^{\gamma_0}(0,t_j;L^2(\Omega))} +  \nm{U(t_j)}_{H_0^1(\Omega)} \\
      \lesssim & \nm{U(0)}_{H_0^1(\Omega)} +
      \nm{f}_{H^{-\gamma_0}(0,t_j;L^2(\Omega))} +
      \snm{\Pi_hu_1}_{H^{\gamma_0}(0,t_j;L^2(\Omega))}.
    \end{aligned}
  \]
  Therefore, \cref{eq:stabi} follows from the following evident estimates:
  \begin{align*}
    \nm{U(0)}_{H_0^1(\Omega)} =
    \nm{R_h u_0}_{H_0^1(\Omega)} \leqslant \nm{ u_0}_{H_0^1(\Omega)}, \\
    \snm{\Pi_hu_1}_{H^{\gamma_0}(0,t_j;L^2(\Omega))}
    \lesssim t_j^{1/2-\gamma_0}\nm{u_1}_{L^2(\Omega)}.
  \end{align*}
  This completes the proof.
\hfill\ensuremath{\blacksquare}\\

To prove \cref{thm:conv}, let us first prove the following two lemmas.
\begin{lem}
  \label{lem:conv}
  If $ u \in H^{1+\gamma_0}(0,T; H_0^1(\Omega)\cap H^2(\Omega)) $ and $ u'' \in
  L^1(0,T;H^2(\Omega)) $, then
  \begin{equation*}
    \label{eq:conv}
    \snm{(U-Q_\tau R_h u)'}_{ H^{\gamma_0}\left( 0,t_j; L^2(\Omega) \right) } +
    \nm{(U-Q_\tau R_h u)(t_j)}_{H_0^1(\Omega)}
    \lesssim \eta_{j,2} + \eta_{j,5}
  \end{equation*}
  for each $ 1 \leqslant j \leqslant J $, where $ \eta_{j,2} $ and $ \eta_{j,5}
  $ are defined as that in \cref{thm:conv}.
\end{lem}
\begin{proof}
  Since
  \[
    \D_{0+}^\gamma\big( u-U-(u-U)(0)-t(I-\Pi_h)u_1 \big) =
    \D_{0+}^{2\gamma_0} \big( (u-U)'-(I-\Pi_h)u_1 \big),
  \]
  combining \cref{eq:weak_sol,eq:algo} yields
  \[
    \dual{
      \D_{0+}^{2\gamma_0} \big( (u-U)' - (I-\Pi_h)u_1 \big),
      \theta'
    }_{H^{\gamma_0}(0,t_j;L^2(\Omega))} +
    \dual{ \nabla (u-U), \nabla\theta' }_{\Omega\times(0,t_j)} = 0,
  \]
  where $ \theta := U-Q_\tau R_h u $. Then, as the definition of $ \Pi_h $
  implies
  \[
    \dual{
      \D_{0+}^{2\gamma_0} (I-\Pi_h)u_1, \theta'
    }_{H^{\gamma_0}(0,t_j;L^2(\Omega))} = 0,
  \]
  we obtain
  \[
    \dual{
      \D_{0+}^{2\gamma_0} \big( u-U \big)', \theta'
    }_{H^{\gamma_0}(0,t_j;L^2(\Omega))} +
    \dual{ \nabla (u-U), \nabla\theta' }_{\Omega\times(0,t_j)} = 0.
  \]
  Therefore, a simple calculation gives
  \[
    \dual{
      \D_{0+}^{2\gamma_0} \theta', \theta'
    }_{H^{\gamma_0}(0,t_j;L^2(\Omega))} +
    \dual{ \nabla \theta, \nabla \theta'}_{ \Omega\times(0,t_j)} =
    E_1 + E_2,
  \]
  where
  \begin{align*}
    E_1 &:= \dual{
      \D_{0+}^{2\gamma_0}(u-Q_\tau R_hu)', \theta'
    }_{H^{\gamma_0}(0,t_j;L^2(\Omega))}, \\
    E_2 &:= \dual{ \nabla(u-Q_\tau R_hu),
    \nabla\theta'}_{\Omega\times(0,t_j)}.
  \end{align*}
  As the fact $ \theta(0)=0 $ implies
  \begin{align*}
    2\dual{ \nabla \theta, \nabla \theta'}_{ \Omega\times(0,t_j)} =
    \nm{\theta(t_j)}_{H_0^1(\Omega)}^2,
  \end{align*}
  by \cref{lem:coer} we obtain
  \begin{equation}
    \label{eq:E1+E2}
    \snm{\theta'}_{H^{\gamma_0}(0,t_j;L^2(\Omega))}^2 +
    \nm{\theta(t_j)}_{H_0^1(\Omega)}^2 \lesssim
    E_1 + E_2.
  \end{equation}

  Next, let us estimate $ E_1 $ and $ E_2 $. As the definition of $ Q_\tau $
  indicates
  \[
    E_1 = \dual{
      \D_{0+}^{2\gamma_0}(u-R_hu)', \theta'
    }_{H^{\gamma_0}(0,t_j;L^2(\Omega))},
  \]
  using \cref{lem:coer} yields
  \begin{equation}\label{eq:E1}
    E_1 \lesssim
    \snm{(I-R_h)u'}_{H^{\gamma_0}(0,t_j;L^2(\Omega))}
    \snm{\theta'}_{H^{\gamma_0}(0,t_j;L^2(\Omega))}.
  \end{equation}
  By the definitions of $ R_h $ and $ Q_\tau $, a straightforward computing
  gives
  \begin{align*}
    E_2 &=
    \dual{
      \nabla u-\nabla(Q_\tau u), \nabla\theta'
      }_{\Omega\times(0,t_j)} = -\dual{
      \Delta u-\Delta(Q_\tau u), \theta'
    }_{\Omega\times(0,t_j)} \\
    &=
    -\dual{
      (I-Q_\tau)\Delta u, \theta'
    }_{\Omega\times(0,t_j)},
  \end{align*}
  so that \cref{lem:I-Q} implies
  \begin{equation} \label{eq:E2}
    \begin{aligned}
      E_2 & \leqslant
      \nm{(I-Q_\tau)\Delta u}_{H^{-\gamma_0}(0,t_j;L^2(\Omega))}
      \snm{\theta'}_{H^{\gamma_0}(0,t_j;L^2(\Omega))} \\
      & \lesssim
      \tau_j^{\min\{\gamma,m-\gamma_0\}}
      \snm{ (I-P_\tau)\Delta u'}_{ H^{\gamma_0}(0,t_j;L^2(\Omega)) }
      \snm{\theta'}_{ H^{\gamma_0}( 0,t_j; L^2(\Omega) ) }.
    \end{aligned}
  \end{equation}

  Finally, by the Young's inequality with $ \epsilon $, combining
  \cref{eq:E1+E2,eq:E1,eq:E2} yields
  \begin{align*}
    {} &
    \snm{\theta'}_{ H^{\gamma_0}( 0,t_j; L^2(\Omega) ) } +
    \nm{\theta(t_j)}_{H_0^1(\Omega) } \\
    \lesssim{} &
    \snm{(I-R_h)u'}_{ H^{\gamma_0}( 0,t_j; L^2(\Omega) ) } +
    \tau_j^{\min\{\gamma,m-\gamma_0\}}
    \snm{ (I-P_\tau)\Delta u'}_{ H^{\gamma_0}(0,t_j;L^2(\Omega)) }.
  \end{align*}
  Therefore, using \cref{lem:frac_appro} proves \cref{eq:conv}.
\end{proof}

\begin{lem}
  \label{lem:73}
  If $ v \in H^{\gamma_0}(0,t) $ with $ 0 < t < \infty $, then
  \begin{equation}
    \label{eq:73}
    \nm{v}_{L^2(0,t)} \leqslant C_\gamma t^{\gamma_0} \snm{v}_{H^{\gamma_0}(0,t)}.
  \end{equation}
\end{lem}
\begin{proof}
  Extending $ v $ to $ \mathbb R \setminus (0,t) $ by zero, by \cref{lem:equiv}
  we have
  \[
    \int_\mathbb R \int_\mathbb R
    \frac{\snm{v(s)-v(\tau)}^2}{\snm{s-\tau}^{1+2\gamma_0}}
    \, \mathrm{d}s \, \mathrm{d}\tau
    \leqslant C_\gamma \snm{v}_{H^{\gamma_0}(\mathbb R)}^2
    = C_\gamma \snm{v}_{H^{\gamma_0}(0,t)}^2.
  \]
  Since a simple computing yields
  \[
    \int_0^t s^{-2\gamma_0} v^2(s) \, \mathrm{d}s
    \leqslant 2\gamma_0
    \int_\mathbb R \int_\mathbb R
    \frac{\snm{v(s)-v(\tau)}^2}{\snm{s-\tau}^{1+2\gamma_0}}
    \, \mathrm{d}s \, \mathrm{d}\tau,
  \]
  we obtain
  \[
    \int_0^t v^2(s) \, \mathrm{d}s <
    t^{2\gamma_0} \int_0^t s^{-2\gamma_0} v^2(s) \, \mathrm{d}s
    \leqslant C_\gamma t^{2\gamma_0} \snm{v}_{H^{\gamma_0}(0,t)}^2.
  \]
  This proves \cref{eq:73} and thus completes the proof.
\end{proof}

\medskip\noindent
{\bf Proof of \cref{thm:conv}.}
As the proof of \cref{eq:conv_1} is trivial by \cref{lem:conv,lem:I-Q}, we only
prove \cref{eq:conv_2}. To do so, we set
  \begin{align*}
    E_1 &:= \nm{(I-R_h)u'}_{L^2(0,t_j;L^2(\Omega))},\\
    E_2 &:= \nm{(U-Q_\tau R_hu)'}_{L^2(0,t_j;L^2(\Omega))}, \\
    E_3 &:= \nm{\big((I-Q_\tau)R_hu\big)'}_{L^2(0,t_j;L^2(\Omega))}.
  \end{align*}
  Since \cref{lem:73} implies
  \begin{align*}
    E_1 & \lesssim t_j^{\gamma_0}
    \snm{(I-R_h)u'}_{H^{\gamma_0}(0,t_j;L^2(\Omega))}
    = t_j^{\gamma_0} \, \eta_{j,2}, \\
    E_2 & \lesssim t_j^{\gamma_0}
    \snm{(U-Q_\tau R_h u)'}_{H^{\gamma_0}(0,t_j;L^2(\Omega))},
  \end{align*}
  by \cref{lem:conv} we obtain
  \[
    E_1 + E_2
   \lesssim \tau_j^{\gamma_0} (\eta_{j,2} + \eta_{j,5}).
  \]
  Also, by \cref{lem:I-Q,lem:ptau},
  \[
    E_3 \lesssim \tau_j^{\gamma_0}
    \snm{\big((I-P_\tau)R_hu'\big)'}_{H^{\gamma_0}(0,t_j;L^2(\Omega))}
    \lesssim \eta_{j,4}.
  \]
   As a consequence,
  \begin{align*}
    E_1 + E_2 + E_3 \lesssim
    \eta_{j,4} + \tau_j^{\gamma_0} \, (\eta_{j,2} + \eta_{j,5}).
  \end{align*}
  Therefore, \cref{eq:conv_2} follows from the estimate
  \[
    \nm{(u-U)'}_{L^2(0,t_j;L^2(\Omega))} \leqslant
    E_1 + E_2 + E_3.
  \]
  This completes the proof.
\hfill\ensuremath{\blacksquare}\\

\section{Numerical Results}
\label{sec:numer}
This section performs some numerical experiments in two-dimensional space to
verify the theoretical results. We set $ \Omega := (0,1)^2 $, $ T := 1 $, and
\[
  u(x,t):=t^rxy(1-x)(1-y), \quad (x,t) \in \Omega \times (0,T),
\]
where $ r > 1 $. In addition, we introduce the following notations:
\[
  \mathcal E_1(U): = \nm{(u-U)'}_{L^2(0,T;L^2(\Omega))}, \quad
  \mathcal E_2(U): = \max_{1\leqslant j\leqslant J}
  \nm{(u-U)(t_j)}_{H_0^1(\Omega)}.
\]

\medskip\noindent{\bf Experiment 1.} This experiment verifies the spatial
accuracy of $ U $ in the case of $ \gamma = 1.5 $. To ensure that the spatial
discretization is dominating, we set $ r = 2 $, $ m = 2 $ and $ J = 64 $. The
numerical results in \cref{tab:ex1} illustrate $ \mathcal E_1(U) = \mathcal
O(h^{n+1}) $ and $ \mathcal E_2(U) = \mathcal O(h^n) $, which agrees well with
\cref{thm:conv_regu}.

\begin{table}[H]
  \caption{$ \gamma = 1.5$, $r = 2$, $m = 2$, $\sigma = 1$, $J = 64$.}
  \label{tab:ex1}
  \small\setlength{\tabcolsep}{3pt}
  \begin{tabular}{ccccccccccc}
    \toprule
    \multirow{2}{*}{$1/h$} & & \multicolumn{4}{c}{$n=1$} &&  \multicolumn{4}{c}{$n=2$} \\
    \cline{3-6} \cline{8-11}
    & & $ \mathcal E_1(U)$ & Order              & $ \mathcal E_2(U)$ & Order
    &                    & $ \mathcal E_1(U)$ & Order              & $ \mathcal E_2(U)$ & Order\\
    \midrule
    8    && 1.59e-3 & --   & 3.02e-2 & --     && 3.68e-5 & --   & 2.11e-3 & --   \\
    16   && 4.03e-4 & 1.98 & 1.52e-2 & 0.99   && 4.59e-6 & 3.00 & 5.31e-4 & 1.99 \\
    32   && 1.01e-4 & 2.00 & 7.60e-3 & 1.00   && 5.73e-7 & 3.00 & 1.33e-4 & 2.00 \\
    64   && 2.53e-5 & 2.00 & 3.80e-3 & 1.00   && 7.17e-8 & 3.00 & 3.32e-5 & 2.00 \\
    \bottomrule
  \end{tabular}
\end{table}

\medskip\noindent{\bf Experiment 2.} This experiment verifies the temporal
accuracy indicated by \cref{thm:conv_regu}. We set $ n=4 $ and $ h=1/16 $ to
ensure that the spatial error is negligible. The numerical results displayed in
\cref{tab:ex2} are summarized as follows.
\begin{itemize}
  \item The accuracy $ \mathcal E_1(U) = \mathcal O(J^{-m}) $ is well verified.
  \item For $ m=1 $, the accuracy $ \mathcal E_2(U) = \mathcal O(J^{-1.2}) $ in
    the case of $ \gamma=1.8 $ is verified, the numerical results about $
    \mathcal E_2(U) $ in the case of $ \gamma=1.5 $ also agree with the
    theoretical accuracy $ \mathcal O((\ln{}J)^{1/2}J^{-3/2}) $, but in the case
    of $ \gamma=1.2 $ the numerical results illustrate $ \mathcal E_2(U) =
    \mathcal O(J^{-2}) $, which exceeds the theoretical accuracy $ \mathcal
    O(J^{-1.5}) $.
  \item For $ m=2 $, the numerical results indicate that $ \mathcal E_2(U) =
    \mathcal O(J^{-3}) $, which exceeds the theoretical accuracy $ \mathcal
    O(J^{-2.5}) $.
\end{itemize}

\begin{table}[H]
  \caption{$r=3$, $n=4$, $h=1/16$, $\sigma=1$.}
  \label{tab:ex2}
  \small\setlength{\tabcolsep}{3pt}
  \begin{tabular}{cccccccccccc}
    \toprule
    \multirow{2}{*}{$\gamma$} &    \multicolumn{5}{c}{$m=1$}  & & \multicolumn{5}{c}{$m=2$}    \\
    \cline{2-6} \cline{8-12}
    & $J$ & $\mathcal E_1(U)$ & Order & $\mathcal E_2(U)$ & Order    &
    & $J$ & $\mathcal E_1(U)$ & Order & $\mathcal E_2(U)$ & Order \\
    \midrule
    \multirow{4}{*}{1.2}
    & 16  & 2.08e-3 & --   & 2.51e-4 & --   &  & 16  & 2.97e-5 & --   & 6.50e-7 & --  \\
    & 32  & 1.04e-3 & 1.00 & 6.12e-5 & 2.03 &  & 32  & 7.48e-6 & 1.99 & 8.02e-8 & 3.02  \\
    & 64  & 5.21e-4 & 1.00 & 1.49e-5 & 2.04 &  & 64  & 1.88e-6 & 1.99 & 1.00e-8 & 3.00  \\
    & 128 & 2.60e-4 & 1.00 & 3.60e-6 & 2.05 &  & 128 & 4.70e-7 & 2.00 & 1.25e-9 & 3.01  \\
    \midrule
    \multirow{4}{*}{1.5}
    & 128  & 2.60e-4 & --   & 7.72e-6 & --   &  & 16  & 3.42e-5 & --   & 1.71e-6 & -- \\
    & 256  & 1.30e-4 & 1.00 & 2.91e-6 & 1.41 &  & 32  & 8.59e-6 & 1.99 & 2.07e-7 & 3.05 \\
    & 512  & 6.51e-5 & 1.00 & 1.07e-6 & 1.44 &  & 64  & 2.15e-6 & 2.00 & 2.52e-8 & 3.04 \\
    & 1024 & 3.26e-5 & 1.00 & 3.90e-7 & 1.46 &  & 128 & 5.38e-7 & 2.00 & 3.08e-9 & 3.03 \\
    \midrule
    \multirow{4}{*}{1.8}
    & 128  & 2.62e-4 & --   & 8.13e-5 & --   &  & 16  & 4.13e-5 & --   & 3.36e-6 & --  \\
    & 256  & 1.31e-4 & 1.00 & 3.60e-5 & 1.18 &  & 32  & 1.03e-5 & 2.00 & 4.05e-7 & 3.05  \\
    & 512  & 6.54e-5 & 1.00 & 1.58e-5 & 1.19 &  & 64  & 2.59e-6 & 2.00 & 4.89e-8 & 3.05  \\
    & 1024 & 3.27e-5 & 1.00 & 6.91e-6 & 1.19 &  & 128 & 6.47e-7 & 2.00 & 5.93e-9 & 3.05  \\
    \bottomrule
  \end{tabular}
\end{table}

\noindent{\bf Experiment 3.} This experiment verifies the temporal accuracy
implied by \cref{thm:singular}. Here we set $ n = 4$, $ h = 1/16 $ so that the
spatial error is sufficiently small. The numerical results are presented in
\cref{tab:ex3-1,tab:ex3-2}. Obviously, the numerical results verifies well that
$ \mathcal E_1(U) = \mathcal O(J^{1/2-r}) $ for $ \sigma=1 $ and that $ \mathcal
E_1(U) = \mathcal O(J^{-m}) $ for $ \sigma>\sigma^* $. For $ m=2 $, the accuracy
$ \mathcal E_2(U)=\mathcal O(J^{-r}) $ in the case of $ \sigma=1 $ is verified,
but $ \mathcal E_2(U) = \mathcal O(J^{-3}) $ in the case of $ \sigma>\sigma^*$
is also observed, which exceeds the theoretical accuracy $ \mathcal O(J^{-2.5})
$. As for $ \mathcal E_2(U) $ in the case of $ m=1 $, we summarize as follows:
in \cref{tab:ex3-1} the accuracies $ \mathcal E_2(U) = \mathcal O(J^{-1.1}) $
for $ \sigma=1 $ and $ \mathcal E_2(U) = \mathcal O(J^{-1.5}) $ for $ \sigma >
  \sigma^* $ are verified; in \cref{tab:ex3-2}, the numerical results agree with
  the theoretical accuracies $ \mathcal E_2(U) = \mathcal O(J^{-1.05}) $ for $
  \sigma=1 $ and $ \mathcal E_2(U) = \mathcal O(J^{-1.4}) $ for $ \sigma >
  \sigma^* $.

\begin{table}[H]
  \caption{$ \gamma = 1.4 $, $n=4$, $h=1/16$.}
  \label{tab:ex3-1}
  \small\setlength{\tabcolsep}{3pt}
  \begin{tabular}{ccccccccccccc}
    \toprule
    \multirow{2}{*}{$\sigma$} &&  \multicolumn{5}{c}{$m=1$, $r=1.1$}  && \multicolumn{5}{c}{$m=2$, $r=1.6$}    \\
    \cline{3-7} \cline{9-13}
    && $J$ & $\mathcal E_1(U)$ & Order & $\mathcal E_2(U)$ & Order    &
    & $J$ & $\mathcal E_1(U)$ & Order & $\mathcal E_2(U)$ & Order \\
    \midrule
    \multirow{4}{*}{1}


    &  & 32  & 4.23e-4 & --   & 7.90e-5 & --   &  & 32  & 4.92e-5 & --   & 3.33e-6 & -- \\
    &  & 64  & 2.80e-4 & 0.59 & 4.15e-5 & 0.93 &  & 64  & 2.30e-5 & 1.10 & 1.02e-6 & 1.70 \\
    &  & 128 & 1.85e-4 & 0.60 & 2.03e-5 & 1.03 &  & 128 & 1.07e-5 & 1.10 & 3.27e-7 & 1.65 \\
    &  & 256 & 1.22e-4 & 0.60 & 9.62e-6 & 1.07 &  & 256 & 5.00e-6 & 1.10 & 1.07e-7 & 1.62 \\

    \midrule
    \multirow{4}{*}{2($\sigma^*$)}
    &  & 32  & 1.38e-4 & --   & 1.13e-5 & --   &  & 16  & 1.57e-5 & --   & 5.63e-7 & -- \\
    &  & 64  & 7.03e-5 & 0.97 & 4.07e-6 & 1.47 &  & 32  & 4.12e-6 & 1.93 & 7.39e-8 & 2.93 \\
    &  & 128 & 3.56e-5 & 0.98 & 1.42e-6 & 1.52 &  & 64  & 1.07e-6 & 1.95 & 9.47e-9 & 2.96 \\
    &  & 256 & 1.80e-5 & 0.99 & 4.89e-7 & 1.54 &  & 128 & 2.75e-7 & 1.96 & 1.20e-9 & 2.98 \\
    \midrule
    \multirow{4}{*}{2.1}
    &  & 32  & 1.34e-4 & --   & 1.08e-5 & --   &  & 16  & 1.54e-5 & --   & 6.34e-7 & --  \\
    &  & 64  & 6.80e-5 & 0.98 & 3.88e-6 & 1.47 &  & 32  & 4.00e-6 & 1.95 & 8.54e-8 & 2.89  \\
    &  & 128 & 3.43e-5 & 0.99 & 1.36e-6 & 1.51 &  & 64  & 1.02e-6 & 1.97 & 1.10e-8 & 2.96  \\
    &  & 256 & 1.72e-5 & 0.99 & 4.68e-7 & 1.54 &  & 128 & 2.59e-7 & 1.98 & 1.38e-9 & 2.99  \\
    \bottomrule
  \end{tabular}
\end{table}

\begin{table}[H]
  \caption{$ \gamma = 1.6$, $n=4$, $h=1/16$.}
  \label{tab:ex3-2}
  \small\setlength{\tabcolsep}{3pt}
  \begin{tabular}{ccccccccccccc}
    \toprule
    \multirow{2}{*}{$\sigma$} &&   \multicolumn{5}{c}{$m=1$, $r=1.15$}  && \multicolumn{5}{c}{$m=2$, $r=1.65$}    \\
    \cline{3-7} \cline{9-13}
    && $J$ & $\mathcal E_1(U)$ & Order             & $\mathcal E_2(U)$ & Order
    &     & $J$               & $\mathcal E_1(U)$ & Order             & $\mathcal E_2(U)$ & Order \\
    \midrule
    \multirow{4}{*}{1}

    && 32  & 5.66e-4  &--   & 1.83e-4  &--   &&     32    & 4.03e-5 & --  & 2.76e-6 & -- \\
    
&& 64  & 3.61e-4  &0.65 & 9.12e-5  &1.01 &&     64    & 1.82e-5 & 1.15  & 8.28e-7 & 1.74 \\
    
&& 128 & 2.29e-4  &0.65 & 4.29e-5  &1.09 &&     128   & 8.19e-6 & 1.15  & 2.58e-7 & 1.68 \\
    
&& 256 & 1.46e-4  &0.65 & 1.97e-5  &1.12 &&    256   & 3.69e-6 & 1.15  & 8.18e-8 & 1.66 \\

    \midrule
    \multirow{4}{*}{2($\sigma^*$)}
    && 32  & 1.92e-4  &--   & 4.61e-5  &--   && 16  & 1.50e-5  &--   & 8.75e-7  & -- \\
    && 64  & 9.60e-5  &1.00 & 1.84e-5  &1.32 && 32  & 3.90e-6  &1.94 & 1.08e-7  &3.01 \\
    && 128 & 4.79e-5  &1.00 & 7.21e-6  &1.35 && 64  & 9.98e-7  &1.97 & 1.36e-8  &2.99 \\
    && 256 & 2.39e-5  &1.00 & 2.78e-6  &1.37 && 128 & 2.55e-7  &1.97 & 1.71e-9  &2.99 \\
    \midrule
    \multirow{4}{*}{2.1}
    && 32  & 1.89e-4  &--   & 4.49e-5  &--   && 16  & 1.50e-5  &--   & 9.68e-7  &--  \\
    && 64  & 9.41e-5  &1.00 & 1.80e-5  &1.32 && 32  & 3.86e-6  &1.96 & 1.25e-7  &2.95  \\
    && 128 & 4.69e-5  &1.01 & 7.05e-6  &1.35 && 64  & 9.81e-7  &1.98 & 1.58e-8  &2.98  \\
    && 256 & 2.33e-5  &1.01 & 2.72e-6  &1.37 && 128 & 2.47e-7  &1.99 & 1.99e-9  &2.99  \\
    \bottomrule
  \end{tabular}
\end{table}

\appendix
\section{Properties of Fractional Calculus Operators}

\begin{lem}[\cite{Samko1993, Diethelm2010,Podlubny1998}]
  \label{lem:basic-frac}
  Let $ -\infty < a < b < \infty $. If $ 0 < \alpha < \beta < \infty $, then
  \begin{align*}
    & \I_{a+}^\alpha \I_{a+}^\beta = \I_{a+}^{\alpha+\beta}, \quad
    \I_{b-}^\alpha \I_{b-}^\beta= \I_{b-}^{\alpha+\beta}, \\
    & \D_{a+}^\beta I_{a+}^\alpha = \D_{a+}^{\beta-\alpha}, \quad
    \D_{b-}^\beta I_{b-}^\alpha = \D_{b-}^{\beta-\alpha}.
  \end{align*}
\end{lem}

\begin{lem}[\cite{Ervin2006}]
  \label{lem:coer}
  Assume that $ -\infty < a < b < \infty $ and $ 0 < \alpha < 1/2 $. If $ v \in
  H^\alpha(a,b) $, then
  \begin{align*}
    & \nm{\D_{a+}^\alpha v}_{L^2(a,b)} \leqslant \snm{v}_{H^\alpha(a,b)}, \\
    & \nm{\D_{b-}^\alpha v}_{L^2(a,b)} \leqslant \snm{v}_{H^\alpha(a,b)}, \\
    & \dual{\D_{a+}^\alpha v, \D_{b-}^\alpha v}_{(a,b)}
    = \cos(\alpha\pi) \snm{v}_{H^{\alpha}(a,b)}^2, \\
    & \dual{\D_{a+}^\alpha v, \D_{b-}^\alpha w}_{(a,b)} \leqslant
    \snm{v}_{H^\alpha(a,b)} \snm{w}_{H^\alpha(a,b)}, \\
    & \dual{\D_{a+}^{2\alpha} v, w}_{H^\alpha(a,b)} =
    \dual{\D_{a+}^\alpha v, \D_{b-}^\alpha w}_{(a,b)} =
    \dual{\D_{b-}^{2\alpha} w, v}_{H^\alpha(a,b)}.
  \end{align*}
\end{lem}

\begin{lem}
  \label{lem:43}
  Suppose that $ -\infty < a < b < \infty $ and $ 0 < \alpha < 1/2 $. If $ v,w \in
  H^\alpha(a,b) $, then
  \[
    \dual{\D_{a+}^\alpha v, \D_{b-}^\alpha w}_{(a,b)} \leqslant
    \snm{v}_{H^\alpha(a,b)} \ssnm{w}_{H^\alpha(a,b)}.
  \]
\end{lem}
\begin{proof}
  By the definition of $ \ssnm{\cdot}_{H^\alpha(a,b)} $, this lemma is a direct
  consequence of \cref{lem:coer}.
\end{proof}

\begin{lem}
  \label{lem:regu}
  If $ \alpha \in [0,1) \setminus \{0.5\} $ and $ 0 < \beta < \infty $, then
  \begin{equation}
    \label{eq:regu}
    \nm{\I_{1-}^\beta v}_{H^{\alpha+\beta}(0,1)} \leqslant
    C_{\alpha,\beta} \nm{v}_{H^\alpha(0,1)}
  \end{equation}
  for all $ v \in H_0^\alpha(0,1) $.
\end{lem}
\begin{proof}
  The proof is a simple modification of that of \cite[Lemma 5.7]{Li2017A}. Let
  us first prove that
  \begin{equation}
    \label{eq:init}
    \nm{\I_{1-}^\beta w}_{H^\beta(0,1)}
    \leqslant C_\beta \nm{w}_{L^2(0,1)}
  \end{equation}
  for all $ w \in L^2(0,1) $ and $ 0 < \beta < 1 $. Extending $ w $ to $
  \mathbb R \backslash (0,1) $ by zero, we define
  \[
    G(t) := \frac1{\Gamma(\beta/2)} \int_t^\infty
    (s-t)^{\beta/2-1} w(s) \, \mathrm{d}s,
    \quad -\infty < t < \infty.
  \]
  Since $ 0 < \beta/2 < 1/2 $, a routine calculation yields $ G \in L^2(\mathbb
  R) $, and then \cite[Theorem~7.1]{Samko1993} implies
  \[
    \mathcal FG(\xi) = (-\mathrm{i}\xi)^{-\beta/2} \mathcal Fw(\xi),
    \quad -\infty < \xi < \infty,
  \]
  where $ \mathcal F: L^2(\mathbb R) \to L^2(\mathbb R) $ is the Fourier
  transform operator, and $ \mathrm{i} $ is the imaginary unit. From the
  well-known Plancherel Theorem it follows
  \[
    \nm{G}_{H^{\beta/2}(\mathbb R)}
    \leqslant C_\beta \nm{w}_{L^2(0,1)},
  \]
  and hence
  \begin{equation}
    \label{eq:regu-3}
    \nm{\I_{1-}^{\beta/2} w}_{H^{\beta/2}(0,1)}
    \leqslant C_\beta \nm{w}_{L^2(0,1)}.
  \end{equation}
  In addition, if $ w \in H_0^1(0,1) $ then, since
  \[
    \D \I_{1-}^{\beta/2} w = -\D \I_{1-}^{\beta/2} \I_{1-} w' =
    \I_{1-}^{\beta/2} w',
  \]
  the estimate \cref{eq:regu-3} implies
  \[
    \nm{ \I_{1-}^{\beta/2}w }_{ H^{1+\beta/2}(0, 1) }
    \leqslant C_\beta \nm{w}_{ H_0^1(0,1)}.
  \]
  Consequently, \cite[Lemma~22.3]{Tartar2007} yields
  \begin{equation}\label{eq:inter}
    \nm{ \I_{1-}^{\beta/2} w }_{ H^\beta(0,1) }
    \leqslant C_\beta \nm{  w }_{ H_0^{\beta/2}(0,1) }
    \quad \text{ for all } w \in H_0^{\beta/2}(0,1).
  \end{equation}
  Therefore, since $\I_{1-}^\beta w = \I_{1-}^{\beta/2} \I_{1-}^{\beta/2} w $,
  combining \cref{eq:regu-3,eq:inter} indicates that \cref{eq:init} holds for
  all $ w \in L^2(0,1) $ and $ 0 < \beta < 1 $.

  Next, let us proceed to prove \cref{eq:regu}. Since the case of $ \beta \in
  \mathbb N $ is trivial, we assume that $ k < \beta < k+1 $ with $ k \in
  \mathbb N $, and so it suffices to prove
  \begin{equation}
    \label{eq:103}
    \nm{\I_{1-}^{\beta-k} v}_{H^{\alpha+\beta-k}(0,1)}
    \leqslant C_{\alpha,\beta} \nm{v}_{H^\alpha(0,1)}.
  \end{equation}
  Since we have already prove that \cref{eq:init} holds for all $ w \in L^2(0,1)
  $ and $ 0 < \beta < 1 $, we obtain
  \begin{align*}
    \nm{\I_{1-}^{\beta-k} w}_{H^{\beta-k}(0,1)} \leqslant
    C_\beta \nm{w}_{L^2(0,1)} \quad \text{ for all } w \in L^2(0,1), \\
    \nm{\I_{1-}^{\beta-k} w}_{H^{1+\beta-k}(0,1)} \leqslant
    C_\beta \nm{w}_{H_0^1(0,1)} \quad \text{ for all } w \in H_0^1(0,1).
  \end{align*}
  Therefore, using \cite[Lemma~22.3]{Tartar2007} again proves \cref{eq:103} and
  thus concludes the proof of this lemma.
\end{proof}

\section{Three Inequalities}
\begin{lem}
  \label{lem:jm}
  Let $ 0\leqslant a < b < \infty $ and $ 0 < \alpha < 1 $. If $ v' \in L^2_{\delta}(a,b) $ with $ 0\leqslant \delta < 1 $ and $ v(b) = 0 $, then
  \begin{align}
    \int_a^b v^2(t) (t-a)^{-\alpha} \, \mathrm{d}t
    \leqslant
    \frac{b^{-\delta}}{(1-\delta)(1-\alpha)} (b-a)^{2-\alpha}
    \nm{v'}_{L_\delta^2(a,b)}^2,
    \label{eq:jm-1} \\
    \int_a^b v^2(t) (b-t)^{-\alpha} \, \mathrm{d}t \leqslant
    \frac{b^{-\delta}}{(1-\delta)(1-\alpha)} (b-a)^{2-\alpha}
    \nm{v'}_{L_\delta^2(a,b)}^2,
    \label{eq:jm-2} \\
    \int_a^b \, \mathrm{d}t \int_a^b
    \snm{v(s)-v(t)}^2 \snm{s-t}^{-1-\alpha} \, \mathrm{d}s
    \leqslant \frac{8b^{-\delta}}{1-\delta}(b-a)^{2-\alpha}
    \nm{v'}_{L_\delta^2(a,b)}^2.
    \label{eq:jm-3}
  \end{align}
\end{lem}
\begin{proof}
  The proof below shall be brief, since the techniques used are standard (see
  Minkowski's integral inequality and Hardy's inequality). For $ a < t < b $, a
  simple computing gives
  \begin{align*}
    {} &
    \snm{v(t)} \leqslant \int_t^b \snm{v'(s)} \, \mathrm{d}s \leqslant
    \left(
      \int_t^b s^{-\delta} \, \mathrm{d}s
    \right)^\frac12
    \left(
      \int_t^b s^\delta \snm{v'(s)}^2 \, \mathrm{d}s
    \right)^\frac12 \\
    \leqslant{} &
    \sqrt{
      \frac{b^{1-\delta} - t^{1-\delta}}{1-\delta}
    }
    \nm{v'}_{L_\delta^2(a,b)} \leqslant
    \sqrt{
      \frac{b^{-\delta}(b-a)}{1-\delta}
    } \nm{v'}_{L_\delta^2(a,b)},
  \end{align*}
  so that we obtain
  \begin{align*}
    {} &
    \int_a^b v^2(t) (t-a)^{-\alpha} \, \mathrm{d}t
    \leqslant   \frac{b^{-\delta}(b-a)}{1-\delta}
    \int_a^b  (t-a)^{-\alpha} \, \mathrm{d}t \,
    \nm{v'}_{L_\delta^2(a,b)}^2 \\
    = {} &
    \frac{b^{-\delta}(b- a)^{2-\alpha}}{(1-\delta)(1-\alpha)}
    \nm{v'}_{L_\delta^2(a,b)}^2,
  \end{align*}
  namely, estimate \cref{eq:jm-1}. Similarly, we have
  \begin{align*}
    {} &
    \int_a^b v^2(t) (b-t)^{-\alpha} \, \mathrm{d}t
    \leqslant\frac{b^{-\delta}(b-a)}{1-\delta}
    \int_a^b  (b-t)^{-\alpha}  \, \mathrm{d}t \,
    \nm{v'}_{L_\delta^2(a,b)}^2 \\
    = {} &
    \frac{b^{-\delta}(b- a)^{2-\alpha}}{(1-\delta)(1-\alpha)}
    \nm{v'}_{L_\delta^2(a,b)}^2,
  \end{align*}
  namely, estimate \cref{eq:jm-2}. Finally, let us prove \cref{eq:jm-3}. Since
  \begin{align*}
    {} &
    \int_a^b \, \mathrm{d}t
    \int_a^b \snm{v(s)-v(t)}^2 \snm{s-t}^{-1-\alpha} \, \mathrm{d}s \\
    ={} &
    2\int_a^b \, \mathrm{d}t
    \int_t^b \snm{\int_t^s v'(\tau) \, \mathrm{d}\tau}^2 (s-t)^{-1-\alpha}
    \, \mathrm{d}s \\
    ={} &
    2\int_a^b \, \mathrm{d}t
    \int_t^b \snm{\int_0^1 v'(t+\theta(s-t)) \, \mathrm{d}\theta}^2
    (s-t)^{1-\alpha} \, \mathrm{d}s \\
    \leqslant{} &
    2(b-a)^{1-\alpha} \int_{a}^{b}
    \left(
      \int_0^1 \sqrt{
        \int_t^b \snm{v'(t+\theta(s-t))}^2 \, \mathrm{d}s
      } \, \mathrm{d}\theta
    \right)^2 \, \mathrm{d} t \\
    ={} &
    2(b-a)^{1-\alpha} \int_a^b
    \left(
      \int_0^1 \sqrt{
        \int_t^{t+\theta(b-t)}
        \snm{v'(\eta)}^2 \theta^{-1} \,\mathrm{d}\eta
      } \,\mathrm{d}\theta
    \right)^2 \, \mathrm{d}t,
  \end{align*}
  the inequality \cref{eq:jm-3} is a direct consequence of
  \begin{align*}
    {} &
    \int_a^b \left(
      \int_0^1 \sqrt{
        \int_t^{t+\theta(b-t)}
        \snm{v'(\eta)}^2 \theta^{-1} \, \mathrm{d}\eta
      } \, \mathrm{d}\theta
    \right)^2 \, \mathrm{d}t \\
    \leqslant{} &
    \int_a^b \left(
      \int_0^1 \sqrt{
        \int_t^{t+\theta(b-t)}
        (\eta/t)^\delta \snm{v'(\eta)}^2 \theta^{-1} \, \mathrm{d}\eta
      } \,\mathrm{d}\theta
    \right)^2 \, \mathrm{d}t \\
    ={} &
    \int_a^b t^{-\delta} \left(
      \int_0^1 \theta^{-1/2} \sqrt{
        \int_t^{t+\theta(b-t)}
        \eta^\delta \snm{v'(\eta)}^2 \, \mathrm{d}\eta
      } \,\mathrm{d}\theta
    \right)^2 \, \mathrm{d}t \\
    \leqslant{} &
    \int_a^b t^{-\delta} \left(
      \int_0^1 \theta^{-1/2}
      \,\mathrm{d}\theta
    \right)^2 \, \mathrm{d}t \, \nm{v'}_{L_\delta^2(a,b)}^2 \\
    \leqslant{} &
    \frac{4b^{-\delta}(b-a)}{1-\delta}
    \nm{v'}_{L_\delta^2(a,b)}^2.
  \end{align*}
  This lemma is thus proved.
\end{proof}



\begin{thebibliography}{10}

    \bibitem{Brenner2008}
    S.~C. Brenner and R.~Scott.
    \newblock {\em The mathematical theory of finite element methods}.
    \newblock Springer-Verlag New York, 3 edition, 2008.

    \bibitem{Wang2014Compact}
    Z. Wang and S. Vong.
    \newblock Compact difference schemes for the modified anomalous fractional
    sub-diffusion equation and the fractional diffusion-wave equation.
    \newblock {\em Journal of Computational Physics}, 277:1--15, 2014.

    \bibitem{Cao2013A}
    J. Cao and C. Xu.
    \newblock A high order schema for the numerical solution of the fractional
    ordinary differential equations.
    \newblock {\em Journal of Computational Physics}, 238(2):154--168, 2013.

    \bibitem{Ciarlet2002}
    P.~G. Ciarlet.
    \newblock {\em The Finite Element Method for Elliptic Problems}.
    \newblock Society for Industrial and Applied Mathematics, 2002.

    \bibitem{Diethelm2010}
    K.~Diethelm.
    \newblock {\em The Analysis of Fractional Differential Equations}.
    \newblock Springer Berlin Heidelberg, 2010.

    \bibitem{Ervin2006}
    V.~J. Ervin and J.~P. Roop.
    \newblock Variational formulation for the stationary fractional advection
    dispersion equation.
    \newblock {\em Numerical Methods for Partial Differential Equations},
    22(3):558--576, 2006.

    \bibitem{Ford2001The}
    N.~J. Ford and A.~C. Simpson.
    \newblock The numerical solution of fractional differential equations: Speed
    versus accuracy.
    \newblock {\em Numerical Algorithms}, 26(4):333--346, 2001.

    \bibitem{Gao2011}
    G.~Gao and Z.~Sun.
    \newblock A compact finite difference scheme for the fractional sub-diffusion
    equations.
    \newblock {\em Journal of Computational Physics}, 230(3):586 -- 595, 2011.

    \bibitem{BolingGuo2015Fractional}
    B.~Guo, X.~Pu, and F.~Huang.
    \newblock {\em Fractional partial differential equations and their numerical
        solutions}.
    \newblock Science Press, 2015.

    \bibitem{Huang2013Two}
    J.~Huang, Y.~Tang, L.~Vázquez, and J.~Yang.
    \newblock Two finite difference schemes for time fractional diffusion-wave
    equation.
    \newblock {\em Numerical Algorithms}, 64(4):707--720, 2013.

    \bibitem{Kilbas2006Theory}
    A.~A.~A. Kilbas, H.~M. Srivastava, and J.~J. Trujillo.
    \newblock Theory and applications of fractional differential equations.
    \newblock 204(49-52):2453--2461, 2006.

    \bibitem{Li2017A}
    B.~Li, H.~Luo, and X.~Xie.
    \newblock A time-spectral algorithm for fractional wave problems.
    \newblock{arXiv:1708.02720}, 2017.

    \bibitem{Li2009}
    X.~Li and C.~Xu.
    \newblock A space-time spectral method for the time fractional diffusion
    equation.
    \newblock {\em SIAM Journal on Numerical Analysis}, 47(3):2108--2131, 2009.

    \bibitem{Liu2011}
    Q.~Liu, F.~Liu, I.~Turner, and V.~Anh.
    \newblock Finite element approximation for a modified anomalous subdiffusion
    equation.
    \newblock {\em Applied Mathematical Modeling}, 35(8):4103--4116, 2011.


    \bibitem{Mclean2015Time}
    W. Mclean and K. Mustapha.
    \newblock Time-stepping error bounds for fractional diffusion problems with
    non-smooth initial data.
    \newblock {\em Journal of Computational Physics}, 293(C):201--217, 2015.

    \bibitem{Meerschaert2012Stochastic}
    M.~M. Meerschaert and A.~Sikorskii.
    \newblock {\em Stochastic Models for Fractional Calculus}.
    \newblock 2012.

    \bibitem{Meerschaert2004Finite}
    M.~M. Meerschaert and C.~Tadjeran.
    \newblock {\em Finite difference approximations for fractional
        advection-dispersion flow equations}.
    \newblock Elsevier Science Publishers B. V., 2004.

    \bibitem{Mustapha2015Time}
    K.~Mustapha.
    \newblock {\em Time-stepping discontinuous Galerkin methods for fractional
        diffusion problems}.
    \newblock Springer-Verlag New York, Inc., 2015.

    \bibitem{Mustapha2012Uniform}
    K.~Mustapha and W.~Mclean.
    \newblock Uniform convergence for a discontinuous galerkin, time-stepping
    method applied to a fractional diffusion equation.
    \newblock {\em Ima Journal of Numerical Analysis}, 32(3):906--925(20), 2012.

    \bibitem{mustapha2014well-posedness}
    K.~Mustapha and D.~Sch\"otzau.
    \newblock Well-posedness of hp-version discontinuous galerkin methods for
    fractional diffusion wave equations.
    \newblock {\em IMA Journal of Numerical Analysis}, 34(4):1426--1446, 2014.

    \bibitem{Mustapha2014A}
    K. Mustapha, B. Abdallah and K. Furati.
    \newblock A discontinuous Petrov-Galerkin method for time-fractional diffusion
    equations.
    \newblock {\em SIAM Journal on Numerical Analysis}, 52(5): 2512--2529, 2014.

    \bibitem{Mustapha2009Discontinuous}
    K. Mustapha and W. Mclean.
    \newblock Discontinuous galerkin method for an evolution equation with a memory
    term of positive type.
    \newblock {\em Mathematics of Computation}, 78(268):1975--1995, 2009.

    \bibitem{Mustapha2012Superconvergence}
    K. Mustapha and W.Mclean.
    \newblock Superconvergence of a discontinuous galerkin method for fractional
    diffusion and wave equations.
    \newblock {\em SIAM Journal on Numerical Analysis}, 51(1):491--515, 2012.

    \bibitem{Podlubny1998}
    I.~Podlubny.
    \newblock {\em Fractional differential equations}.
    \newblock Academic Press, 1998.


    \bibitem{Thomee1997}
    V.~Thom\'ee.
    \newblock {\em Galerkin Finite Element Methods for Parabolic Problems}.
    \newblock Springer-Verlag Berlin Heidelberg, 1997.

    \bibitem{Ren2017}
    J.~Ren, X.~Long, S.~Mao, and J.~Zhang.
    \newblock Superconvergence of finite element approximations for the fractional
    diffusion-wave equation.
    \newblock {\em Journal of Scientific Computing}, pages 1--19, 2017.

    \bibitem{Samko1993}
    S.~G. Samko, A.~A. Kilbas, and O.~I. Marichev.
    \newblock {\em Fractional integrals and derivatives: theory and applications}.
    \newblock USA: Gordon and Breach Science Publishers, 1993.

    \bibitem{sun2006fully}
    Z.~Sun and X.~Wu.
    \newblock A fully discrete difference scheme for a diffusion-wave system.
    \newblock {\em Applied Numerical Mathematics}, 56(2):193--209, 2006.

    \bibitem{Tartar2007}
    L.~Tartar.
    \newblock {\em An introduction to Sobolev spaces and interpolation spaces}.
    \newblock Springer Berlin Heidelberg, 2007.

    \bibitem{yang2016spectral}
    Y.~Yang, Y.~Chen, Y.~Huang, and H.~Wei.
    \newblock Spectral collocation method for the time-fractional diffusion-wave
    equation and convergence analysis.
    \newblock {\em Computers and Mathematics with Applications}, 73(6):1218--1232,
    2017.

    \bibitem{Yuste2006}
    S.~B. Yuste.
    \newblock Weighted average finite difference methods for fractional diffusion
    equations.
    \newblock {\em Journal of Computational Physics}, 216(1):264--274, 2006.

    \bibitem{Yuste2005}
    S.~B. Yuste and L.~Acedo.
    \newblock An explicit finite difference method and a new von neumann-type
    stability analysis for fractional diffusion equations.
    \newblock {\em SIAM Journal on Numerical Analysis}, 42(5):1862--1874, 2005.

    \bibitem{G2002Chaos}
    G.~M. Zaslavsky.
    \newblock Chaos, fractional kinetics, and anomalous transport.
    \newblock {\em Physics Reports}, 371(6):461--580, 2002.

    \bibitem{Zayernouri2014Fractional}
    M. Zayernouri, G. E. Karniadakis.
    \newblock Fractional spectral collocation method.
    \newblock {\em SIAM Journal on Scientific Computing}, 36(1):A40--A62, 2014.

    \bibitem{Zayernouri2012Karniadakis}
    M.~Zayernouri and G.~E. Karniadakis.
    \newblock Discontinuous spectral element methods for time- and space-fractional
    advection equations.
    \newblock {\em SIAM Journal on Scientific Computing}, 36(4), B684-B707, 2012.

    \bibitem{Zayernouri2014Exponentially}
    M.~Zayernouri and G.~E. Karniadakis.
    \newblock Exponentially accurate spectral and spectral element methods for
    fractional odes.
    \newblock {\em Journal of Computational Physics}, 257(2):460--480, 2014.

    \bibitem{Zhao2014Short}
    L. Zhao and W. Deng.
    \newblock Short memory principle and a predictor–corrector approach for
    fractional differential equations.
    \newblock {\em Journal of Computational \& Applied Mathematics}, 40(1):137--165,
    2014.

    \bibitem{Zheng2015}
    M.~Zheng, F.~Liu, I.~Turner, and V.~Anh.
    \newblock A novel high order space-time method for the time fractional
    fokker-planck equation.
    \newblock {\em SIAM Journal on Scientific Computing}, 37(2):A701--A724, 2015.
 \end{thebibliography}



\end{document}